%% file: RK_boundary_conditions.tex
\RequirePackage{forest}
\RequirePackage{pgfplots}
\RequirePackage{tikz-cd}

\documentclass[sn-mathphys]{sn-jnl}

\newif\ifreport
\reporttrue

\ifreport
\usepackage{csl}
\fi
\usepackage{cleveref}
\usepackage{mathtools}
\usepackage{mleftright}
\usepackage{physics}
\usepackage{siunitx}
\usepackage{subcaption}
\usepackage{xparse}

\theoremstyle{thmstyleone}
\newtheorem{theorem}{Theorem}

\theoremstyle{thmstyletwo}
\newtheorem{remark}{Remark}

\input{header}

\jyear{2021}

\begin{document}
	
\title[Eliminating Order Reduction with GARK Methods]{Eliminating Order Reduction on Linear, Time-Dependent ODEs with GARK Methods}
\author*[1]{\fnm{Steven} \sur{Roberts}}\email{roberts115@llnl.gov}
\author[2]{\fnm{Adrian} \sur{Sandu}}\email{sandu@cs.vt.edu}

\affil*[1]{\orgdiv{Center for Applied Scientific Computing}, \orgname{Lawrence Livermore National Laboratory}, \orgaddress{\city{Livermore}, \state{California}, \country{USA}}}

\affil[2]{\orgdiv{Computational Science Laboratory, Department of Computer Science}, \orgname{Virginia Tech}, \orgaddress{\street{2202 Kraft Dr.}, \city{Blacksburg}, \postcode{24060}, \state{Virginia}, \country{USA}}}

\abstract{%
When applied to stiff, linear differential equations with time-dependent forcing, Runge--Kutta methods can exhibit convergence rates lower than predicted by the classical order condition theory.  Commonly, this order reduction phenomenon is addressed by using an expensive, fully implicit Runge--Kutta method with high stage order or a specialized scheme satisfying additional order conditions.  This work develops a flexible approach of augmenting  an arbitrary Runge--Kutta method with a fully implicit method used to treat the forcing such as to maintain the classical order of the base scheme.  Our methods and analyses are based on the general-structure additive Runge--Kutta framework.  Numerical experiments using diagonally implicit, fully implicit, and even explicit Runge--Kutta methods confirm that the new approach eliminates order reduction for the class of problems under consideration, and the base methods achieve their theoretical orders of convergence.}

\keywords{Order reduction, Runge--Kutta, General-structure additive Runge--Kutta methods, Convergence analysis}

\pacs[MSC Classification]{65L04, 65L20}

\ifreport
\cslauthor{Steven Roberts and Adrian Sandu}
\cslyear{22}
\cslreportnumber{1}
\cslemail{roberts115@llnl.gov, sandu@cs.vt.edu}
\csltitlepage
\fi

\maketitle

\input{intro}

\input{formulation}
\input{order_conditions}

\input{pr}
\input{advection}
\input{heat_eq}
\input{connections}
\input{conclusion}

\begin{appendices}
	\input{appendix}
\end{appendices}

\bibliography{main.bib}

\section*{Statements and Declarations}

\subsection*{Funding}

This work was performed under the auspices of the U.S. Department of Energy by Lawrence Livermore National Laboratory under Contract DE-AC52-07NA27344 and was supported by the U.S. Department of Energy, Office of Science, Office of Advanced Scientific Computing Research.  LLNL-JRNL-830949

This document was prepared as an account of work sponsored by an agency of the United States government. Neither the United States government nor Lawrence Livermore National Security, LLC, nor any of their employees makes any warranty, expressed or implied, or assumes any legal liability or responsibility for the accuracy, completeness, or usefulness of any information, apparatus, product, or process disclosed, or represents that its use would not infringe privately owned rights. Reference herein to any specific commercial product, process, or service by trade name, trademark, manufacturer, or otherwise does not necessarily constitute or imply its endorsement, recommendation, or favoring by the United States government or Lawrence Livermore National Security, LLC. The views and opinions of authors expressed herein do not necessarily state or reflect those of the United States government or Lawrence Livermore National Security, LLC, and shall not be used for advertising or product endorsement purposes.

Steven Roberts (in part) and Adrian Sandu were supported by awards NSF ACI--1709727, NSF CDS\&E--MSS 1953113, DE-SC0021313, and by the Computational Science Laboratory at Virginia Tech.

\subsection*{Competing Interests}

The authors have no relevant financial or non-financial interests to disclose.

\subsection*{Author Contributions}

Both authors contributed to the method formulation, order conditions, and convergence analysis.  Numerical experiments and the first draft of the manuscript were prepared by Steven Roberts.  Both authors read and approved the final manuscript.

\subsection*{Data Availability}

The datasets generated during and/or analysed during the current study are available from the corresponding author on reasonable request.

\end{document}

%% file: header.tex
\crefname{equation}{}{}
\Crefname{equation}{Equation}{Equations}

\graphicspath{{./Figures/}}
\newcommand{\datadir}{./Data}

\pgfplotsset{compat=1.16}
\pgfplotscreateplotcyclelist{mycolor}{%
	blue,mark=*\\%
	red,mark=square*\\%
	green,mark=triangle*\\%
	orange,mark=star\\%
}
\usepgfplotslibrary{groupplots}
\pgfplotsset{every axis/.append style={
	cycle list name=mycolor,
	label style={font=\scriptsize},
	legend style={font=\scriptsize,legend columns=-1,/tikz/every even column/.append style={column sep=1em}},
	tick label style={font=\scriptsize},
	width=\linewidth
}}

\forestset{
	btree/.style={
		for tree={
			grow=north,
			minimum width=5pt,
			minimum height=5pt,
			inner sep=0pt,
			outer sep=0pt,
			s sep=3pt,
			l=4pt,
			l sep=4pt,
			if n children=0{}{
				if level=0{
					baseline
				}{}
			}
		},
		delay={
			for tree={
				if={strequal("1",content)}{fill=white,draw,circle,content={}}{},
				if={strequal("2",content)}{fill=black,circle,content={}}{},
				if={strequal("3",content)}{fill=white,draw,content={}}{},
				if={strequal(".",content)}{content={$\dots$},no edge}{},
				if={strequal(":",content)}{content={$\dots$}}{}
			}
		}
	}
}
\newcommand{\btree}[1]{\Forest{btree #1}}

\DeclareMathOperator{\diag}{diag}

\newenvironment{butchertableau}[2][1.35]{\def\arraystretch{#1}\array{#2}}{\endarray}
\newcommand{\btentry}[1]{\scriptstyle{#1}}

\NewDocumentCommand{\eye}{g g}{\mathbf{I}\IfValueT{#1}{_{{#1} \times \IfValueTF{#2}{#2}{#1}}}}
\NewDocumentCommand{\zero}{g g}{0\IfValueT{#1}{_{{#1} \IfValueT{#2}{\times {#2}}}}}
\NewDocumentCommand{\one}{g g}{\mathbf{1}\IfValueT{#1}{_{{#1} \IfValueT{#2}{\times {#2}}}}}
\NewDocumentCommand{\unitvec}{m}{\mathbf{e}_{#1}}
\RenewDocumentCommand{\R}{g g}{\mathbb{R}^{\IfValueT{#1}{#1} \IfValueT{#2}{\times {#2}}}}
\NewDocumentCommand{\Cplx}{g g}{\mathbb{C}^{\IfValueT{#1}{#1} \IfValueT{#2}{\times {#2}}}}

\NewDocumentCommand{\comp}{m O{} s}{^{ \left\{ #1 \right\} #2 \IfBooleanT{#3}{T} }}

\newcommand{\nvar}{d}
\newcommand{\lte}{\textnormal{lte}}
\newcommand{\tree}{\mathfrak{t}}
\newcommand{\utree}{\mathfrak{u}}

\renewcommand{\A}{\mathbf{A}}
\renewcommand{\b}{\mathbf{b}}
\renewcommand{\c}{\mathbf{c}}

%% file: intro.tex
\section{Introduction}
\label{sec:Order_Reduction:intro}

Consider the linear, constant-coefficient, inhomogeneous system of ordinary differential equations
\begin{equation} \label{eq:linear_inhomogeneous_ode}
	y' = f(t, y) = L y + g(t), 
	\quad
	y(t_0) = y_0,
	\quad
	t \in [t_0, t_f],
\end{equation}
where $y(t) \in \Cplx{\nvar}$.  Problems of this form frequently arise from the spatial discretization of linear partial differential equations (PDEs).  In this case, $L \in \Cplx{\nvar \times \nvar}$ approximates spatial differential operators and $g(t)$ accounts for time-dependent source terms and boundary conditions.

Runge--Kutta methods are widely used to integrate \cref{eq:linear_inhomogeneous_ode}.  One step of an $s$-stage Runge--Kutta method using timestep $h = t_{n+1}-t_{n}$ is given by \cite{kutta1901beitrag}
\begin{subequations} \label{eq:rk}
	\begin{align}
		\label{eq:rk:stages}
		Y_i &= y_{n} + h \sum_{j=1}^s a_{i,j} f(t_n + c_j h, Y_j), \qquad i = 1, \dots, s, \\
		\label{eq:rk:step}
		y_{n+1} &= y_{n} + h \sum_{j=1}^s b_{j} f(t_n + c_j h, Y_j),
	\end{align}
\end{subequations}
and its coefficients are concisely represented by the Butcher tableau
\begin{equation} \label{eq:rk_tableau}
	\begin{butchertableau}{c|c}
		c & A \\ \hline
		& b^T
	\end{butchertableau}.
\end{equation}

An important special case of \cref{eq:linear_inhomogeneous_ode} is the Prothero--Robinson (PR) test problem \cite{prothero1974stability}:
\begin{equation} \label{eq:pr}
	y' = \lambda (y - \phi(t)) + \phi'(t),
	\qquad
	y(0) = \phi(0).
\end{equation}
In their seminal work, Prothero and Robinson analyzed the error and stability of Runge--Kutta methods applied to \cref{eq:pr} as $\Re(h \lambda) \to -\infty$ and $h \to \zero$.  For this seemingly innocuous problem, the order of convergence for a Runge--Kutta method may be lower than what is predicted by classical order condition theory: a phenomenon referred to as \textit{order reduction}.  Classical order condition theory \cite[Sections I.7 and II.3]{hairer1993solving} typically requires $f$ to have a moderate Lipschitz constant that is independent of the timestep $h$, and for the PR problem, this assumption does not hold.  The analysis of order reduction has been extended to many other classes of problems including linear PDEs \cite{sanz1986convergence,verwer1986convergence,ostermann1992runge} and general, nonlinear problems \cite{frank1981concept,burrage1986study}.

It is well-known that the Runge--Kutta simplifying assumptions
\begin{subequations} \label{eq:simplifying_assumptions}
	\begin{alignat}{4}
		B(p):&
		\quad
		& b^T c^{k-1} &= \frac{1}{k},
		& \quad
		k &= 1, \dots, p, \label{eq:simplifying_assumptions:B} \\
		C(q):&
		\quad
		& A c^{k-1} &= \frac{c^k}{k},
		& \quad
		k &= 1, \dots, q, \label{eq:simplifying_assumptions:C}
	\end{alignat}
\end{subequations}
mitigate the order reduction phenomenon \cite[Section IV.15]{hairer1996solving}.  A method satisfying $B(p)$ and $C(q)$ with $p \geq q$ is said to have stage order $q$.  Ideally, a method would have the stage order equal to the classical order, but in many cases, this cannot be achieved.  Explicit Runge--Kutta methods, for example, have a maximum stage order of one, while diagonally implicit methods have a maximum stage order of two.  The concept of weak stage order (WSO) has been explored in \cite{ketcheson2020dirk}.  As the name suggests, it considers weaker but sufficient conditions to avoid order reduction.
In \cite{ostermann1992runge,ostermann1993rosenbrock}, the authors derive a rigorous error expansion and order conditions for stiff, parabolic PDEs.  Similar results have been derived for the PR problem in \cite{rang2014analysis,rang2016prothero}.

An approach used to address order reduction in initial boundary value problems is a modified treatment of the boundary conditions in the stages \cref{eq:rk:stages} \cite{abarbanel1996removal,pathria1997correct,alonso2002runge,alonso2004avoiding}.  Many of these utilize time derivatives of the boundary conditions which would not be required in a traditional Runge--Kutta stage.  One can interpret this as a composite method where a Runge--Kutta method is used to treat the differential operators, and a multi-derivative scheme is used to treat the boundary conditions.

In this work, we propose integrating \cref{eq:linear_inhomogeneous_ode} using a general-structure additive Runge--Kutta (GARK) method \cite{sandu2015generalized} to eliminate order reduction.  For an arbitrary ``base'' Runge--Kutta method used to treat the linear term $L y$, we derive a different, fully implicit Runge--Kutta scheme used to treat the forcing $g(t)$.  Our approach does not increase the number of linear solves per step nor does it require time derivatives of $g(t)$.  In many cases, we are able to \textit{reduce} the number of $g(t)$ evaluations per step compared to the base method.  Furthermore, unlike stage order conditions, there are no restrictions on the order for explicit or diagonally implicit method structures.  In fact, no order conditions need to be imposed on the base method beyond classical order conditions; order reduction is mitigated by imposing order conditions on the companion method.  

The remainder of this paper is organized as follows.  In \cref{sec:Order_Reduction:formulation}, the new GARK-based methods for solving \cref{eq:linear_inhomogeneous_ode} are derived.  \Cref{sec:Order_Reduction:order_conditions} develops the error analysis and order condition theory. \Cref{sec:Order_Reduction:pr,sec:Order_Reduction:advection,sec:Order_Reduction:heat_eq} provide three numerical experiments that test the convergence properties of Runge--Kutta methods and their GARK extensions and validate the new methodology.  Connections to previous work on alleviating order reduction are explained in 
\cref{sec:Order_Reduction:connections}.  Finally, the findings of the paper are summarized in \cref{sec:Order_Reduction:conclusion}.

%% file: formulation.tex
\section{Method Formulation}
\label{sec:Order_Reduction:formulation}

We will consider a splitting of \cref{eq:linear_inhomogeneous_ode} into the linear term and the time-dependent forcing term:
\begin{equation} \label{eq:ode_split}
	y' = \underbrace{L y}_{\eqqcolon f\comp{1}(t, y)} + \underbrace{g(t)}_{\eqqcolon f\comp{2}(t, y)}.
\end{equation}
A general, two-way partitioned GARK scheme solves \cref{eq:ode_split} as follows \cite{sandu2015generalized}:
\begin{equation} \label{eq:gark}
	\begin{alignedat}{2}
		Y\comp{1}_i &= y_n + h \sum_{j=1}^{s\comp{1}} a\comp{1,1}_{i,j} L Y\comp{1}_j + h \sum_{j=1}^{s\comp{2}} a\comp{1,2}_{i,j} g\mleft( t_n + c\comp{2}_j h \mright), & \qquad & i = 1, \ldots s\comp{1}, \\
		Y\comp{2}_i &= y_n + h \sum_{j=1}^{s\comp{1}} a\comp{2,1}_{i,j} L Y\comp{1}_j + h \sum_{j=1}^{s\comp{2}} a\comp{2,2}_{i,j} g \mleft( t_n + c\comp{2}_j h \mright), & \qquad & i = 1, \ldots s\comp{2}, \\
		y_{n+1} &= y_n + h \sum_{j=1}^{\comp{1}} b\comp{1}_j L Y\comp{1}_j + h \sum_{j=1}^{s\comp{2}} b\comp{2}_j g \mleft( t_n + c\comp{2}_j h \mright).
	\end{alignedat}
\end{equation}
In contrast to \cref{eq:rk}, the stages in \cref{eq:gark} are partitioned, there are four sets of $A$ coefficient matrices used in the stages, and two sets of $b$ coefficients.  Note that the computation of $y_{n+1}$ in \cref{eq:gark} does not involve $Y\comp{2}_i$.  With $Y\comp{1}_i$ serving as the only useful stages, it may appear that \cref{eq:gark} degenerates into an additive Runge--Kutta (ARK) method which does not have partitioned stages.  This is not the case, however, as the GARK formalism allows the additional flexibility of treating the linear term and forcing terms with a different number of stages.  That is, $\A\comp{1,2}$ can be rectangular.

We can simplify and rewrite \cref{eq:gark} in the compact form
\begin{subequations} \label{eq:gark_simplified}
	\begin{align}
		\label{eq:gark_simplified:stages}
		Y\comp{1} &= \one{s\comp{1}} \otimes y_n + \left( \A\comp{1,1} \otimes Z \right) Y\comp{1} + h \left( \A\comp{1,2} \otimes \eye{\nvar} \right) g \mleft( t_n + \c\comp{2} h \mright), \\
		\label{eq:gark_simplified:step}
		y_{n+1} &= y_{n} + \left( \b\comp{1}* \otimes Z \right) Y\comp{1} + h \left( \b\comp{2}* \otimes \eye{\nvar} \right) g \mleft( t_n + \c\comp{2} h \mright),
	\end{align}
\end{subequations}
where $\otimes$ denotes the Kronecker product, $\one{s\comp{1}}$ is a vector of ones of dimension $s\comp{1}$, and $Z = h L$.  We also use the notation
\begin{equation*}
	Y\comp{1} \coloneqq \begin{bmatrix}
		Y\comp{1}_1 \\
		\vdots \\
		Y\comp{1}_{s\comp{1}}
	\end{bmatrix},
	\qquad
	g \mleft( t_n + \c\comp{2} h \mright) \coloneqq \begin{bmatrix}
		g \mleft( t_n + c\comp{2}_1 h \mright) \\
		\vdots \\
		g \mleft( t_n + c\comp{2}_{s\comp{2}} h \mright) \\
	\end{bmatrix}.
\end{equation*}
We  represent the simplified method \cref{eq:gark_simplified} compactly with the tableau
\begin{equation} \label{eq:tableau_simplified}
	\begin{butchertableau}{c|c}
		\c\comp{1}* & \c\comp{2}* \\ \hline
		\A\comp{1,1} & \A\comp{1,2} \\ \hline
		\b\comp{1}* & \b\comp{2}*
	\end{butchertableau}.
\end{equation}
We refer to $(\A\comp{1,1}, \b\comp{1}, \c\comp{1})$ as the \textit{base method}.  The coefficients $(\A\comp{1,2}, \b\comp{2}, \c\comp{2})$ do not necessarily define a Runge--Kutta method because $\A\comp{1,2}$ can be retangular.  Nevertheless, we refer to it as the \textit{companion method}.

The implicitness of \cref{eq:gark_simplified} is entirely determined by the structure of $\A\comp{1,1}$.  With $f\comp{2}$ only a function of time, we can make $\A\comp{1,2}$ a fully-dense matrix without incurring additional function evaluations or linear solves.

%% file: order_conditions.tex
\section{Order Conditions}
\label{sec:Order_Reduction:order_conditions}

For the error analysis in this section, we will assume that $g(t)$ is at least $p$-times differentiable.  Thus, the exact solution $y(t)$ is $p+1$-times differentiable.

\subsection{Classical Order Conditions}

For sufficiently small $h$, we can rely on existing, tree-based GARK order condition theory to analyze the local truncation error of the new method \cref{eq:gark_simplified}.  We present these order conditions in the following theorem which is proved in Appendix \ref{app:classical_order_conditions}.

\begin{theorem}[Classical order conditions]
	\label{tmh:classical_order_conditions}
	The method \cref{eq:gark_simplified} applied to \cref{eq:linear_inhomogeneous_ode} has classical order $p$ if and only if
	\begin{subequations} \label{eq:nonstiff_order_conditions}
		\begin{alignat}{2}
			\label{eq:nonstiff_order_conditions:linear}
			\b\comp{1}* \left( \A\comp{1,1} \right)^{k-1} \one{s\comp{1}} &= \frac{1}{k!},
			\qquad & 1 &\leq k \leq p \\
			\label{eq:nonstiff_order_conditions:bushy}
			\b\comp{2}* \c\comp{2}[\times (k-1)] &= \frac{1}{k},
			\qquad & 1 &\leq k \leq p \\
			\label{eq:nonstiff_order_conditions:palm}
			\b\comp{1}* \left( \A\comp{1,1} \right)^{k-1} \A\comp{1,2} \c\comp{2}[\times (\ell-1)] &= \frac{(\ell - 1)!}{(\ell + k)!},
			\qquad & 1 &\leq k, \ell \text{ and } k + \ell \leq p,
		\end{alignat}
		where ``$\times$'' in an exponent indicates an element-wise power of a vector.
	\end{subequations}
\end{theorem}

\subsection{Stiff Order Conditions}

Before the asymptotic regime is reached, a method satisfying \cref{eq:nonstiff_order_conditions} may exhibit an order of convergence less than $p$.  In order to characterize this behavior we need to reexamine the local truncation error produced at each step and the accumulated global error.

We define the global errors at step $n$ in the stages \cref{eq:gark_simplified:stages} and solution \cref{eq:gark_simplified:step} to be
\begin{equation}
	\label{eqn:global-error-definition}
	\begin{split}
		E_n &\coloneqq y \mleft( t_n + \c\comp{1} h \mright) - Y\comp{1} = \begin{bmatrix}
			y \mleft( t_n + c\comp{1}_1 h \mright) - Y\comp{1}_1 \\
			\vdots \\
			y \mleft( t_n + c\comp{1}_{s\comp{1}} h \mright) - Y\comp{1}_{s\comp{1}}
		\end{bmatrix} \\
		e_{n} &\coloneqq y(t_n) - y_n,
	\end{split}
\end{equation}
respectively.

In \cref{eq:gark_simplified:stages}, we replace the initial condition $y_n$ with $y(t_n)$ and replace the stages $Y\comp{1}$ with $y \mleft( t_n + \c\comp{1} h \mright)$:
\begin{equation} \label{eq:stage_errs}
	\begin{split}
		y \mleft( t_n + \c\comp{1} h \mright) &= \one{s\comp{1}} \otimes y(t_n) + \left( \A\comp{1,1} \otimes Z \right) y \mleft( t_n + \c\comp{1} h \mright) \\
		& \quad + h \left( \A\comp{1,2} \otimes \eye{\nvar} \right) g \mleft( t_n + \c\comp{2} h \mright) + \Delta_n \\
		&= \one{s\comp{1}} \otimes y(t_n) + \left( \A\comp{1,1} \otimes Z \right) y \mleft( t_n + \c\comp{1} h \mright) \\
		& \quad + h \left( \A\comp{1,2} \otimes \eye{\nvar} \right) y' \mleft( t_n + \c\comp{2} h \mright) \\
		& \quad - \left( \A\comp{1,2} \otimes Z \right) y \mleft( t_n + \c\comp{2} h \mright) + \Delta_n.
	\end{split}
\end{equation}
In general, the exact solution does not exactly satisfy the stage equations, and thus, there is a stage defect $\Delta_n$.  A Taylor expansion yields
\begin{equation*}
	\begin{split}
		\Delta_n &= \sum_{k=1}^{p}  \frac{\c\comp{1}[k] - k \A\comp{1,2} \c\comp{2}[\times (k-1)]}{k!} \otimes \left( h^k y^{(k)}(t_n) \right) \\
		& \quad + \sum_{k=0}^{p} \frac{\A\comp{1,2} \c\comp{2}[\times k] - \A\comp{1,1} \c\comp{1}[\times k]}{k!} \otimes \left( h^k Z y^{(k)}(t_n) \right) + R_n.
	\end{split}
\end{equation*}
The term $R_n$ contains the remainder terms from the Taylor series.  For brevity, we will defer writing the full form of these terms until the end of the derivation.  We use $R_n$ instead of a more typical $\order{h^{p+1}}$ because the residuals depend on $Z$, which, in turn, can depend on $h$ in arbitrary ways.

The global error in the stages follows by subtracting \cref{eq:gark_simplified:stages} from \cref{eq:stage_errs}:
\begin{align*}
	E_n &= \one{s\comp{1}} \otimes e_n + \left( \A\comp{1,1} \otimes Z \right) E_n + \Delta_n \\
	&= \left( \eye{\nvar s\comp{1}} - \A\comp{1,1} \otimes Z \right)^{-1} \left( \one{s\comp{1}} \otimes e_n + \Delta_n \right).
\end{align*}
Now we can repeat the process for the step \cref{eq:gark_simplified:step} by substituting $y(t_n)$ for $y_n$ and $y(t_{n+1})$ for $y_{n+1}$:
\begin{equation} \label{eq:step_errs}
	\begin{split}
		y(t_{n+1}) &= y(t_n) + \left( \b\comp{1}* \otimes Z \right) y \mleft( t_n + \c\comp{1} h \mright) - \left( \b\comp{2}* \otimes Z \right) y \mleft( t_n + \c\comp{2} h \mright) \\
		& \quad + h \left( \b\comp{2}* \otimes \eye{\nvar} \right) y' \mleft( t_n + \c\comp{2} h \mright) + \delta_n, \\
		\delta_n &= \sum_{k=1}^{p} \frac{1 - k \b\comp{2}* \c\comp{2}[\times (k-1)]}{k!} h^k y^{(k)}(t_n) \\
		& \quad + \sum_{k=0}^{p} \frac{\b\comp{2}* \c\comp{2}[\times k] - \b\comp{1}* \c\comp{1}[\times k]}{k!} h^k Z y^{(k)}(t_n) + r_n.
	\end{split}
\end{equation}
Note the step defect $\delta_n$ also contains a remainder term $r_n$.

The global error recurrence is obtained by subtracting \cref{eq:gark_simplified:step} from \cref{eq:step_errs}.  This yields
\begin{equation} \label{eq:err_recurrence}
	e_{n+1} = e_n + \left( \b\comp{1}* \otimes Z \right) E_n + \delta_n = R\comp{1}(Z) e_n + \lte_n,
\end{equation}
where the linear stability function is given by
\begin{equation} \label{eq:stability_function}
	R\comp{1}(z) = 1 + z \b\comp{1} \left( \eye{s\comp{1}} - z \A\comp{1,1} \right) \one{s\comp{1}}.
\end{equation}
The local truncation error at step $n$ is
\begin{subequations} \label{eq:lte_def}
	\begin{align}
		\label{eq:lte_def:exact}
		\lte_n &= \left( \b\comp{1}* \otimes Z \right) \left( \eye{\nvar s\comp{1}} - \A\comp{1,1} \otimes Z \right)^{-1} \Delta_n + \delta_n \\
		\label{eq:lte_def:series}
		&= \sum_{k=0}^{p} W_k(Z) \frac{h^k}{k!} y^{(k)}(t_n) + Q_n,
	\end{align}
\end{subequations}
with the Taylor series remainders combining to give
\begin{align*}
	Q_n &= \left( \b\comp{1}* \otimes Z \right) \left( \eye{\nvar s\comp{1}} - \A\comp{1,1} \otimes Z \right)^{-1} R_n + r_n \\
	&= \left( y^{(p+1)}(\xi_n) + \widetilde{W}_{p+1}(Z) y^{(p+1)}(\boldsymbol{\zeta}_n) \right) \frac{h^{p+1}}{(p+1)!},
\end{align*}
for some $\xi_n \in (t_n, t_{n+1})$ and $\boldsymbol{\zeta}_n \in \R{s\comp{2}}$ with $\zeta_{n,i}$ between $t_n$ and $t_n + c\comp{2}_i h$ for $i = 1, \dots, s\comp{2}$.
The local error coefficients are defined as
\begin{equation} \label{eq:lte_residuals}
	\begin{split}
		W_0(z) &= z \left( \b\comp{2}* \one{s\comp{2}} - \b\comp{1}* \one{s\comp{1}} \right) \\
		& \quad + z^2 \b\comp{1}* \big( \eye{s\comp{1}} - z \A\comp{1,1} \big)^{-1} \left( \A\comp{1,2} \one{s\comp{2}} - \A\comp{1,1} \one{s\comp{1}} \right), \\
		%
		W_k(z) &= 1 + \widetilde{W}_{k}(z) \one{s\comp{2}}, \\
		\widetilde{W}_k(z) &= \left( \b\comp{2}* + z \b\comp{1}* \big( \eye{s\comp{1}} - z \A\comp{1,1} \big)^{-1} \A\comp{1,2} \right) \left( z \big( \mathbf{C}\comp{2} \big)^k \right. \\
		& \quad \left. - k \big( \mathbf{C}\comp{2} \big)^{k-1} \right), \qquad k \geq 1,
	\end{split}
\end{equation}
where $\mathbf{C}\comp{2} = \diag \mleft( c\comp{2}_1, \dots, c\comp{2}_{s\comp{2}} \mright)$.  For simplicity, both \cref{eq:stability_function,eq:lte_residuals} have been written in a scalar form but are rational matrix functions of $Z$.

We can expand \cref{eq:lte_def:series} further by expressing it as a multivariate series in $h$ and $Z$:
\begin{equation} \label{eq:lte_series}
	\lte_n = \sum_{k=0}^{p} \sum_{\ell=0}^{\infty} w_{k,\ell} \frac{h^k Z^\ell}{k!} y^{(k)}(t_n) + Q_n.
\end{equation}
The coefficients $w_{k,\ell}$ are found by taking a Maclaurin series of $W_k(z)$:
\begin{equation} \label{eq:residual_coeffs}
	w_{k,\ell} = \begin{cases}
		0, & k = 0, \ell = 0, \\
		\b\comp{2}* \one{s\comp{2}} - \b\comp{1}* \one{s\comp{1}}, & k = 0, \ell = 1, \\
		\b\comp{1}* \left( \A\comp{1,1} \right)^{\ell-2} \left( \A\comp{1,2} \one{s\comp{2}} - \A\comp{1,1} \one{s\comp{1}} \right), & k = 0, \ell > 1, \\
		1 - k \b\comp{2}* \c\comp{2}[\times (k-1)], & k > 0, \ell = 0, \\
		\b\comp{2}* \c\comp{2}[\times k] - k \b\comp{1}* \A\comp{1,2} \c\comp{2}[\times (k-1)], & k > 0, \ell = 1, \\
		\b\comp{1}* \left( \A\comp{1,1} \right)^{\ell-2} \\
		\quad \cdot \left( \A\comp{1,2} \c\comp{2}[\times k] - k \A\comp{1,1} \A\comp{1,2} \c\comp{2}[\times (k-1)] \right), & k > 0, \ell > 1.
	\end{cases}
\end{equation}

For a nonstiff problem, that is, when $Z = \order{h}$ and $L$ is independent of $h$, we can simplify \cref{eq:lte_series} to
\begin{align*}
	\lte_n &= \sum_{k=0}^{p} \sum_{\ell=0}^{\infty} w_{k,\ell} h^{k+\ell} L^{\ell} y^{(k)}(t_n) + Q_n \\
	&= \sum_{k=0}^{p} \left( \sum_{\ell=0}^{k} w_{\ell, k - \ell} L^{k-\ell} y^{(\ell)}(t_n) \right) h^k + Q_n.
\end{align*}
Note that the terms $L^{\ell} y^{(k)}(t_n)$ can be linearly independent for arbitrarily large values of $k$ and $\ell$.  In general, we must set $w_{k,\ell} = 0$ for $k + \ell \leq p$ to achieve $\lte_n = \order{h^{p+1}}$.  When $w_{k,\ell}$ are viewed as elements of an infinite-dimensional matrix, a triangular region must be zero.  More formally, the equivalence of nonstiff order conditions is summarized in the following diagram:
\begin{equation*}
	\begin{tikzcd}[arrows=Leftrightarrow,sep=large]
		\lte_n = \order{h^{p+1}} \arrow{r} \arrow{d} & w_{k,\ell} = 0 \text{ for } k, \ell \geq 0 \text{ and } k + \ell \leq p \arrow{d} \\
		\text{\cref{eq:nonstiff_order_conditions}} \arrow{r} & W_k(z) = \order{h^{p+1-k}} \text{ for } 0 \leq k \leq p
	\end{tikzcd}
\end{equation*}
As expected, we can recover the tree-based order conditions given in \cref{eq:nonstiff_order_conditions} from \cref{eq:lte_series}.

For stiff problems, however, $h$ and $Z$ can have a more complex relationship, and $W_k(z)$ are not necessarily bounded by powers of $h$. Consequently, more stringent order conditions are necessary: we have to completely cancel out $W_k(z)$ for $k = 0, \dots, p$.  Using the infinite-dimensional matrix interpretation of $w_{k,\ell}$, we must set a rectangular region to zero instead of a triangular region.
\begin{theorem} \label{thm:zero_residual}
	$W_k(z) \equiv 0$ if and only if $w_{k,\ell} = 0$ for $\ell = 0, \dots, s\comp{1} + 1$.
\end{theorem}

\begin{proof}
	Note that we can express the local error coefficients \cref{eq:lte_residuals} as
	\begin{equation} \label{eq:w_rational}
		W_k(z) = \frac{n_{k,0} + n_{k,1} z + \dots + n_{k,s\comp{1} + 1} z^{s\comp{1} + 1}}{d_{k,0} + d_{k,1} z + \dots + d_{k,s\comp{1}} z^{s\comp{1}}} = \sum_{\ell=0}^{\infty} w_{k,\ell} z^{\ell},
	\end{equation}
	with $d_{k,0} \neq 0$.
	
	($\Leftarrow$) Suppose that $w_{k,\ell} = 0$ for $\ell = 0, \dots, s\comp{1} + 1$.  From \cref{eq:w_rational}, $n_{k,i} = \sum_{j=0}^{\min(i, s\comp{1})} w_{k,i-j} d_{k,j} = 0$ for $i = 0, \dots, s\comp{1}+1$.  Thus, $W_k(z) \equiv 0$.
	
	($\Rightarrow$) For the order direction of the proof, it is clear that if $W_k(z) \equiv 0$, the Maclaurin series coefficients $w_{k,\ell} = 0$ for $\ell \geq 0$.  
\end{proof}


Following the idea of Prothero and Robinson, we can also examine the behavior of local truncation error \cref{eq:lte_def} when $Z \to -\infty$.  In this limit, we cannot rely on the power series expansion in $Z$ used in \cref{eq:lte_series}; instead, we consider a Laurent series in $Z$:
\begin{equation*}
	\lte_n = \sum_{k=0}^{p} \sum_{\ell=-1}^{\infty} x_{k,\ell} \frac{h^k Z^{-\ell}}{k!} y^{(k)}(t_n) + Q_n.
\end{equation*}
To ensure this series exists, we require $\b\comp{1}*$ to be in the rowspace of $\A\comp{1,1}$ and any zero eigenvalues of $\A\comp{1,1}$ to be regular.  Thus, there exists a $\mathbf{v} \in \R{s\comp{1}}$ such that $\mathbf{v}^T \A\comp{1,1} = \b\comp{1}*$, and $\A\comp{1,1}$ has a Jordan decomposition of the form $\A\comp{1,1} = \mathbf{S}^{-1} \begin{bsmallmatrix} \boldsymbol{\Lambda} & 0 \\ 0 & 0 \end{bsmallmatrix} \mathbf{S}$, where $\mathbf{\Lambda}$ contains the Jordan blocks for nonzero eigenvalues.  As an intermediate step in the expansion of local error coefficients \cref{eq:lte_residuals}, note that
\begin{equation} \label{eq:intermediate_series}
	\begin{split}
		z \b\comp{1}* \big( \eye{s\comp{1}} - z \A\comp{1,1} \big)^{-1}
		&= \mathbf{v}^T \mathbf{S}^{-1} \begin{bmatrix}
			z \boldsymbol{\Lambda} \left( \eye - z \boldsymbol{\Lambda} \right)^{-1} & 0 \\
			0 & 0 \cdot \eye
		\end{bmatrix} \mathbf{S} \\
		&= z \b\comp{1}* \mathbf{S}^{-1} \begin{bmatrix}
		-\sum_{\ell=1}^{\infty} \left(z \boldsymbol{\Lambda} \right)^{-\ell} & 0 \\
		0 & 0
		\end{bmatrix} \mathbf{S} \\
		&= -\b\comp{1}* \sum_{\ell=0}^{\infty} z^{-\ell} \boldsymbol{\Omega}^{\ell+1},
	\end{split}
\end{equation}
where $\boldsymbol{\Omega} = \left( \A\comp{1,1} \right)^D$ is the Drazin inverse \cite{drazin1958pseudo} of $\A\comp{1,1}$.  In particular, $\boldsymbol{\Omega} = \left( \A\comp{1,1} \right)^{-1}$ when $\A\comp{1,1}$ is invertible, but the Drazin inverse also accounts for methods with explicit stages.  Substituting \cref{eq:intermediate_series} into \cref{eq:lte_residuals} yields:
\begin{equation} \label{eq:laurent_coeffs}
	x_{k,\ell} = \begin{cases}
		\b\comp{1}* \boldsymbol{\Omega}^{\ell+2} \left( \A\comp{1,1} \one{s\comp{1}} - \A\comp{1,2} \one{s\comp{2}} \right), & k = 0, \ell \geq 0, \\
		\left( \b\comp{2}* - \b\comp{1}* \boldsymbol{\Omega} \A\comp{1,2} \right) \c\comp{2}[\times k], & k \geq 0, \ell = -1, \\
		1 - k \left( \b\comp{2}* - \b\comp{1}* \boldsymbol{\Omega} \A\comp{1,2} \right) \c\comp{2}[\times (k-1)] \\
		\quad - \b\comp{1}* \boldsymbol{\Omega}^2 \A\comp{1,2} \c\comp{2}[\times k], & k > 0, \ell = 0, \\
		\b\comp{1}* \boldsymbol{\Omega}^{\ell+2} \left( \A\comp{1,2} \c\comp{2}[\times k] \right. \\
		\quad \left. - k \A\comp{1,1} \A\comp{1,2} \c\comp{2}[\times (k-1)] \right), & k > 0, \ell > 0.
	\end{cases}
\end{equation}
Unless $x_{k,-1} = 0$ for $k \geq 0$, $\lte_n$ diverges as $\abs{Z} \to \infty$.  \Cref{eq:laurent_coeffs} suggests the following sufficient condition to ensure $W_k(z)$ is bounded away from its poles:
\begin{equation} \label{eq:bounded_simplifying_assumption}
	\b\comp{2}* = \b\comp{1}* \boldsymbol{\Omega} \A\comp{1,2}.
\end{equation}
If $\A\comp{1,1}$ is invertible and the GARK scheme \cref{eq:gark_simplified} is stiffly accurate \cite[Definition 3.3]{sandu2015generalized}, that is
\begin{equation}
	\unitvec{s\comp{1}}^T \A\comp{1,1} = \b\comp{1}*
	\quad \text{and} \quad
	\unitvec{s\comp{1}}^T \A\comp{1,2} = \b\comp{2}*,
\end{equation}
then \cref{eq:bounded_simplifying_assumption} is automatically satisfied.

\ifreport
\subsection{Simplifying Assumptions}

Extensions of traditional Runge--Kutta simplifying assumptions \cref{eq:simplifying_assumptions} to the GARK framework have been proposed in \cite{tanner2018generalized}.  Quadrature simplifying assumptions are defined as
\begin{equation} \label{eq:GARK_B_simplify}
	B\comp{\sigma}(p): \quad \b\comp{\sigma}* \c\comp{\sigma}[\times (k-1)] = \frac{1}{k}, \quad k = 1, \dots, p.
\end{equation}
A method satisfying $B\comp{\sigma}(1)$ for all $\sigma$ is said to be consistent with \cref{eq:linear_inhomogeneous_ode}.  This condition is both necessary and sufficient for classical first order convergence.  The stage order simplifying assumption in \cref{eq:simplifying_assumptions:C} extends to
\begin{equation} \label{eq:GARK_C_simplify}
	C\comp{\sigma,\mu}(q): \quad \A\comp{\sigma,\mu} \c\comp{\mu}[\times (k-1)] = \frac{\c\comp{\sigma}[\times k]}{k}, \quad k = 1, \dots, q.
\end{equation}
The commonly-used internal consistency assumption \cite[Definition 2.3]{sandu2015generalized} is equivalent to $C\comp{\sigma,\mu}(1)$ for all $\sigma$ and $\mu$.

\begin{theorem} \label{thm:simplifying_assumptions}
	Suppose the GARK method \cref{eq:gark_simplified} has coefficients satisfying the simplifying assumptions $B\comp{\sigma}(p)$ and $C\comp{1,\sigma}(q)$ for $\sigma = 1, 2$.  Then $W_k(z) \equiv 0$ for $k = 0, \dots, \min(p, q) - 1$.
\end{theorem}

\begin{proof}
	Assume $B\comp{\sigma}(p)$ and $C\comp{1,\sigma}(q)$ hold for $\sigma = 1, 2$.  Since $W_0(z)$ has a different form than the other residuals in \cref{eq:lte_residuals}, we will treat is separately.  One can easily verify that $W_0(z) \equiv 0$ when $p,q \geq 1$.  For $k = 1, \dots, \min(p, q) - 1$ we have
	\begin{align*}
		W_k(z) &= \left( 1 - k \b\comp{2}* \c\comp{2}[\times (k-1)] \right) + \left( \b\comp{2}* \c\comp{2}[\times k] - k \b\comp{1}* \A\comp{1,2} \c\comp{2}[\times (k-1)] \right) z \\
		& \quad + \sum_{\ell = 2}^{\infty} \b\comp{1}* \left( \A\comp{1,1} \right)^{\ell-2} \left( \A\comp{1,2} \c\comp{2}[\times k] - k \A\comp{1,1} \A\comp{1,2} \c\comp{2}[\times (k-1)] \right) z^{\ell} \\
		&= \left( \b\comp{2}* \c\comp{2}[\times k] - \b\comp{1}* \c\comp{1}[\times k] \right) z \\
		& \quad + \sum_{\ell = 2}^{\infty} \b\comp{1}* \left( \A\comp{1,1} \right)^{\ell-2} \left( \A\comp{1,2} \c\comp{2}[\times k] - \A\comp{1,1} \c\comp{1}[\times k] \right) z^{\ell} \\
		&= 0.
	\end{align*}
	by \cref{eq:GARK_B_simplify,eq:GARK_C_simplify}.
\end{proof}

The result in \cref{thm:simplifying_assumptions} is slightly weaker than what can be achieved with unpartitioned Runge--Kutta methods.  When we cast a Runge--Kutta method as a GARK method, the result in \cref{thm:simplifying_assumptions} can be sharpened by one order, i.e., $W_{\min(p,q)}(z) \equiv 0$.  From \cref{thm:simplifying_assumptions}, we also see that the minimal conditions of consistency and internal consistency imply $W_0(z) \equiv 0$.

\fi

\subsection{Global Error and Convergence}

Following \cite[Section IV.15]{hairer1996solving}, the accumulation of local truncation errors into the global error $e_n$ is found by unrolling the error recurrence given in \cref{eq:err_recurrence}:
\begin{equation*}
	e_{n+1} = \left( R\comp{1}(Z) \right)^{n+1} e_0 + \sum_{j=0}^n \left( R\comp{1}(Z) \right)^{n-j} \lte_j.
\end{equation*}
\begin{theorem} \label{thm:global_err}
	 Let $\langle\cdot,\cdot\rangle$ denote an inner product on $\Cplx{\nvar}$ and $\norm{\cdot}$ denote the induced norm.  Assume the linear operator in \cref{eq:linear_inhomogeneous_ode} satisfies
	 \begin{equation} \label{eq:L_property}
	 \forall y: \quad \Re \langle y, L y \rangle \leq \mu \norm{y}^2, \qquad \mu \leq 0.
	 \end{equation}
	 If the GARK scheme \cref{eq:gark_simplified} has an A-stable base method and satisfies
	 \begin{subequations} \label{eq:order_conditions}
		 \begin{align}
		 	\label{eq:order_conditions:W}
		 	&0 = w_{k,\ell}, \text{ for } k = 0, \dots, p \text{ and } \ell = 0, \dots, s\comp{1} + 1, \\
		 	\label{eq:order_conditions:bounded}
		 	&\widetilde{W}_{p+1}(z) \text{ uniformly bounded over } \{ z \in \Cplx \, : \, \Re(z) \leq \mu \},
		 \end{align}
	 \end{subequations}
	 then there exists a positive constant $C$ such that for $t_f = t_0 + n h$ fixed, the global error is bounded by
	 \begin{equation} \label{eq:global_err}
	 	\norm{e_{n}} \leq C \max_{t \in T} \norm{y^{(p+1)}(t)} h^p,
	 \end{equation}
	 where $C$ is a constant depending only on $\mu$, the timespan, and the method coefficients.  The set $T$ is the timespan enlarged if abscissae lie outside the standard range:
	 \begin{equation*}
	 	T = \left[ t_0 + \min \mleft( 0, c\comp{2}_1, \dots, c\comp{2}_{s\comp{2}} \mright) h, t_f + \max \mleft( 0, c\comp{2}_1 - 1, \dots, c\comp{2}_{s\comp{2}} - 1 \mright) h \right].
	 \end{equation*}
\end{theorem}

\begin{remark}
	When \cref{eq:linear_inhomogeneous_ode} is stiff, the exact solution can have an initial phase of rapid exponential decay.  During this time, $y^{(p+1)}(t)$ in \cref{eq:global_err} can become disproportionally large.  This is a consideration not just for our GARK methods but for all B-convergent Runge--Kutta schemes as well.  Outside of the initial transient phase, the derivatives of $y$ can be bounded by a moderately-sized value.
\end{remark}

\begin{proof}
	By the assumptions of A-stability and \cref{eq:L_property}, we can apply Theorem 4 from \cite[Section 2]{hairer1982stability} to show
	\begin{equation*}
		\norm{R\comp{1}(Z)} \leq \sup_{\Re(z) \leq \mu} \abs{R\comp{1}(z)} \leq 1. \\
	\end{equation*}
	\Cref{eq:order_conditions:bounded} implies there must exist finite constants $\ell_i$ depending only on $\mu$ and method coefficients such that
	\begin{equation*}
		\norm{\widetilde{W}_{p+1,i}(Z)} \leq \sup_{\Re(z) \leq \mu} \abs{\widetilde{W}_{p+1,i}(z)} = \ell_i, \qquad i = 1, \dots, s\comp{2}.
	\end{equation*}
	With the help of \cref{eq:lte_def,thm:zero_residual}, \cref{eq:order_conditions:W} implies
	\begin{align*}
		\norm{\lte_n}
		&= \norm{y^{(p+1)}(\xi_n) + \widetilde{W}_{p+1}(Z) y^{(p+1)}(\boldsymbol{\zeta}_n)} \frac{h^{p+1}}{(p+1)!} \\
		&\leq \frac{1 + \sum_{i=1}^{s\comp{2}} \ell_i}{(p+1)!} \max_{t \in T} \norm{y^{(p+1)}(t)} h^{p+1}.
	\end{align*}
	Thus, the global error satisfies the inequality
	\begin{align*}
		\norm{e_{n}}
		&\leq \sum_{j=0}^{n-1} \norm{R\comp{1}(Z)}^{n-1-j} \norm{\lte_j} \\
		&\leq \sum_{j=0}^{n-1} \norm{\lte_j} \\
		&\leq \underbrace{\frac{t_f - t_0}{(p+1)!} \left( 1 + \sum_{i=1}^{s\comp{2}} \ell_i \right)}_{C} \max_{t \in T} y^{(p+1)}(t) h^p.
	\end{align*}
	Note that $C$ is a finite constant with the desired dependencies.
\end{proof}

%% file: pr.tex
\section{Numerical Schemes and Empirical Prothero--Robinson Convergence}
\label{sec:Order_Reduction:pr}

In this section, we will examine the error and convergence properties of singly diagonally implicit Runge--Kutta (SDIRK) methods applied to
\begin{equation} \label{eq:pr_cos}
	y' = -200 \big(y - \cos(t) \big) - \sin(t), \quad y(0) = 1, \quad t \in [0, 1].
\end{equation}
This is a special case of the PR test problem \cref{eq:pr} with $\lambda = -200$ and $\phi(t) = \cos(t)$.

\subsection{Order Two}
\label{sec:Order_Reduction:pr:SDIRK_2}

First, we will start with the popular, second order, L-stable SDIRK method \cite{alexander1977diagonally}
\begin{equation} \label{eq:SDIRK2}
	\begin{butchertableau}{c|cc}
		 \btentry{1} - \frac{1}{\sqrt{2}} &  \btentry{1} - \frac{1}{\sqrt{2}} & \btentry{0} \\
		\btentry{1} & \frac{1}{\sqrt{2}} &  \btentry{1} - \frac{1}{\sqrt{2}} \\ \hline
		& \frac{1}{\sqrt{2}} &  \btentry{1} - \frac{1}{\sqrt{2}}
	\end{butchertableau},
\end{equation}
which we will refer to as SDIRK2.  Substituting its coefficients into \cref{eq:lte_def,eq:lte_residuals} reveals that
\begin{equation} \label{eq:lte_SDIRK2}
	\lte_n = \frac{(4-3 \sqrt{2}) Z}{2 ((\sqrt{2}-2) Z+2)^2} h^2 y''(t_n) + \frac{(7-5 \sqrt{2}) Z-3 \sqrt{2}+4}{6 ((\sqrt{2}-2) Z+2)^2} h^3 y^{(3)}(t_n) + \cdots.
\end{equation}
For a nonstiff ODE, $Z = \order{h}$, and $\lte_n = \order{h^3}$.  If we take $Z \to -\infty$ the differential equation becomes an algebraic equation and $\lte_n = 0$.  Between these extremes, there are ``moderately stiff'' problems for which the leading term of \cref{eq:lte_SDIRK2} can cause order reduction.

In order to eliminate this problematic second order error, we extend SDIRK2 to a GARK method \cref{eq:gark_simplified} such that $W_k(z) \equiv 0$ for $k = 0, 1, 2$.  This introduces the new coefficients $\A\comp{1,2}$, $\b\comp{2}$, and $\c\comp{2}$.  We make the somewhat arbitrary choice $\c\comp{2} = [0,\frac{1}{2},1]$.  We impose the stiff accuracy property  $\b\comp{2}* = \unitvec{s\comp{1}}^T \A\comp{1,2}$.  Using \cref{thm:zero_residual}, the unspecified coefficients in $\A\comp{1,2}$ are uniquely determined by the order conditions
\begin{equation*}
	w_{k,\ell} = 0, \text{ for } k = 0, 1, 2 \text{ and } \ell = 0, 1, 2, 3.
\end{equation*}
Solving them, we arrive at the following method with tableau \cref{eq:tableau_simplified}:
\begin{equation} \label{eq:SDIGARK2}
	\begin{butchertableau}{cc|ccc}
		 \btentry{1} - \frac{1}{\sqrt{2}} & \btentry{1} & \btentry{0} & \frac{1}{2} & \btentry{1} \\ \hline
		 \btentry{1} - \frac{1}{\sqrt{2}} & \btentry{0} & \frac{13}{2}-\frac{9}{\sqrt{2}} & \btentry{10 \sqrt{2}-14} & \frac{17}{2}-\btentry{6 \sqrt{2}} \\
		\frac{1}{\sqrt{2}} &  \btentry{1} - \frac{1}{\sqrt{2}} &  \btentry{2\sqrt{2} }-\frac{5}{2} &  \btentry{6-4 \sqrt{2} } &  \btentry{2 \sqrt{2} }-\frac{5}{2} \\ \hline
		\frac{1}{\sqrt{2}} &  \btentry{1} - \frac{1}{\sqrt{2}} & \btentry{2 \sqrt{2} } -\frac{5}{2} &  \btentry{6-4 \sqrt{2} } &  \btentry{2\sqrt{2} }-\frac{5}{2}
	\end{butchertableau},
\end{equation}
We name this scheme SDIGARK2, and it has
\begin{equation*}
	\lte_n = \frac{(3- 2 \sqrt{2}) Z-12 \sqrt{2} +16}{6 ((\sqrt{2}-2) Z+2)^2} h^3 y^{(3)}(t_n) + \cdots.
\end{equation*}
\Cref{fig:pr_convergence_SDIRK2} shows numerical results of SDIRK2 and SDIGARK2 when applied to \cref{eq:pr_cos}.  We can see that SDIRK2 suffers from order reduction, while SDIGARK2 maintains an order of convergence of at least two.

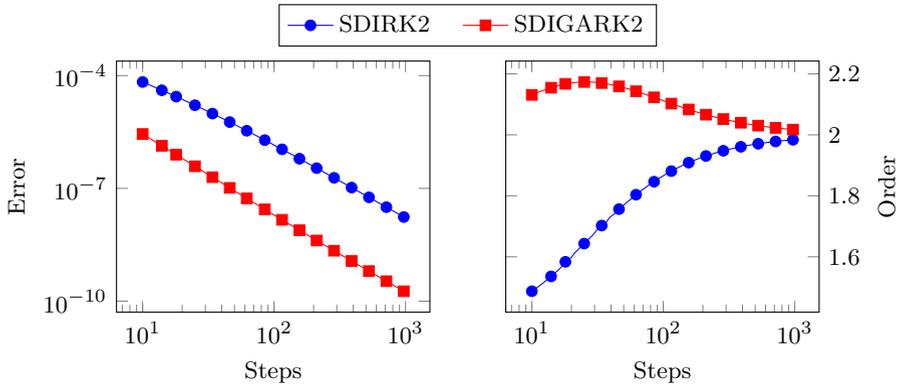
\begin{figure}[ht!]
	\centering
	\begin{tikzpicture}
		\begin{groupplot}[group style={group size=2 by 1},xlabel={Steps},width=0.48\textwidth]
			\nextgroupplot[xmode=log,ymode=log,ylabel={Error},mark repeat={4},legend to name=leg:pr_SDIRK2]
			\addplot table[x=Steps,y=SDIRK22 Error,col sep=comma]{\datadir/pr.csv}; \addlegendentry{SDIRK2}
			\addplot table[x=Steps,y=SDIGARK22a Error,col sep=comma]{\datadir/pr.csv}; \addlegendentry{SDIGARK2}
			\coordinate (c1) at (rel axis cs:0,1);
			\nextgroupplot[xmode=log,ylabel={Order},mark repeat={4},yticklabel pos=right]
			\addplot table[x=Steps,y=SDIRK22 Order,col sep=comma]{\datadir/pr.csv};
			\addplot table[x=Steps,y=SDIGARK22a Order,col sep=comma]{\datadir/pr.csv};
			\coordinate (c2) at (rel axis cs:1,1);
		\end{groupplot}
		\coordinate (c3) at ($(c1)!.5!(c2)$);
		\node[above] at (c3 |- current bounding box.north) {\pgfplotslegendfromname{leg:pr_SDIRK2}};
	\end{tikzpicture}
	\caption{Convergence and order for the methods \cref{eq:SDIRK2,eq:SDIGARK2} when applied to the PR problem \cref{eq:pr_cos}.}
	\label{fig:pr_convergence_SDIRK2}
\end{figure}

\subsection{Order Three}
\label{sec:Order_Reduction:pr:SDIRK_3}

In contrast to \cref{eq:SDIRK2}, the third order method we consider next is neither L-stable nor stiffly accurate.  The method is named SDIRK3 and has the tableau \cite{norsett1974semi}
\begin{equation} \label{eq:SDIRK3}
	\begin{butchertableau}{c|cc}
		\frac{\sqrt{3}+3}{6} & \frac{\sqrt{3}+3}{6} & \btentry{0} \\
		\frac{3-\sqrt{3}}{6} & -\frac{1}{\sqrt{3}} & \frac{\sqrt{3}+3}{6} \\ \hline
		& \frac{1}{2} & \frac{1}{2}
	\end{butchertableau}.
\end{equation}
With a local truncation error of
\begin{equation*}
	\lte_n = \frac{(2 \sqrt{3}+3) Z^2}{2 ((\sqrt{3}+3) Z-6)^2} h^2 y''(t_n) + \frac{(3 \sqrt{3}+5) Z^2}{6 ((\sqrt{3}+3) Z-6)^2} h^3 y^{(3)}(t_n) + \cdots,
\end{equation*}
order reduction is expected outside of the $Z = \order{h}$ regime.  We derive a GARK version of \cref{eq:SDIRK3} following a similar methodology to the one used in \cref{sec:Order_Reduction:pr:SDIRK_2}, but select $\c\comp{2} = \begin{bmatrix}-2 & -1 & 0 & 1\end{bmatrix}$.  For constant stepsizes, this choice only requires one evaluation of $g$ per step because $g(t_n - 2h), \dots, g(t_n)$ were already computed in previous steps.  One can view this as treating the linear term of \cref{eq:linear_inhomogeneous_ode} with SDIRK3 and the forcing term with a linear multistep method.  Our new method SDIGARK3a, given by
\begin{equation} \label{eq:SDIGARK3}
	\begin{butchertableau}{cc|cccc}
		\frac{\sqrt{3}+3}{6} & \frac{3-\sqrt{3}}{6} & \btentry{-2} & \btentry{-1} & \btentry{0} & \btentry{1} \\ \hline
		\frac{\sqrt{3}+3}{6} & \btentry{0} & \frac{-3 \sqrt{3}-5}{36} & \frac{11 \sqrt{3}+18}{36} & \frac{-13 \sqrt{3}-15}{36} & \frac{11 \sqrt{3}+20}{36} \\
		-\frac{1}{\sqrt{3}} & \frac{\sqrt{3}+3}{6} & \frac{7 \sqrt{3}+13}{36} & \frac{-25 \sqrt{3}-48}{36} & \frac{29 \sqrt{3}+75}{36} & \frac{-17 \sqrt{3}-22}{36} \\ \hline
		\frac{1}{2} & \frac{1}{2} & \frac{\sqrt{3}+3}{36} & \frac{-\sqrt{3}-4}{12} & \frac{\sqrt{3}+11}{12} & \frac{12-\sqrt{3}}{36}
	\end{butchertableau},
\end{equation}
has $W_k(z) \equiv 0$ for $k = 0, \dots, 3$ so that
\begin{equation*}
	\lte_n = \frac{\left(2 \sqrt{3}+5\right) Z+2 \sqrt{3}+3}{2 \left(\left(\sqrt{3}+3\right) Z-6\right)^2} h^4 y^{(4)}(t_n) + \cdots.
\end{equation*}
\Cref{fig:pr_convergence_SDIRK3} plots the errors produced by various third order schemes when integrating \cref{eq:pr_cos}.  It confirms order reduction for SDIRK3, and interestingly, the convergence line for SDIGARK3a has a cusp around $250$ steps.  While the error is still consistent with the bounds from \cref{thm:global_err}, the instantaneous order of convergence dips below three following the cusp.  This occurs because $W_4(z)$ has a root at $z \approx -0.7637$, and around this point, the leading error term no longer dominates the local truncation error.  We note SDIGARK2 did not have this behavior because its root of $W_3(z)$ is positive.

\begin{figure}[ht!]
	\centering
	\begin{tikzpicture}
		\begin{groupplot}[group style={group size=2 by 1},xlabel={Steps},width=0.48\textwidth]
			\nextgroupplot[xmode=log,ymode=log,ylabel={Error},mark repeat={4},legend to name=leg:pr_SDIRK3]
			\addplot table[x=Steps,y=SDIRK32 Error,col sep=comma]{\datadir/pr.csv}; \addlegendentry{SDIRK3}
			\addplot table[x=Steps,y=SDIGARK32a Error,col sep=comma]{\datadir/pr.csv}; \addlegendentry{SDIGARK3a}
			\addplot table[x=Steps,y=SDIGARK32b Error,col sep=comma]{\datadir/pr.csv}; \addlegendentry{SDIGARK3b}
			\addplot table[x=Steps,y=WSO33 Error,col sep=comma]{\datadir/pr.csv}; \addlegendentry{WSO3 \cref{eq:DIRK3_WSO3}}
			\draw[-latex] (axis cs:8e1,1e-12) node[left]{\small{$w_4(Z) = 0$}} to (axis cs:2.3e2,1.6e-13);
			\coordinate (c1) at (rel axis cs:0,1);
			\nextgroupplot[xmode=log,ylabel={Order},mark repeat={4},ymin=1,ymax=5,yticklabel pos=right]
			\addplot table[x=Steps,y=SDIRK32 Order,col sep=comma]{\datadir/pr.csv};
			\addplot table[x=Steps,y=SDIGARK32a Order,col sep=comma]{\datadir/pr.csv};
			\addplot table[x=Steps,y=SDIGARK32b Order,col sep=comma]{\datadir/pr.csv};
			\addplot table[x=Steps,y=WSO33 Order,col sep=comma]{\datadir/pr.csv};
			\coordinate (c2) at (rel axis cs:1,1);
		\end{groupplot}
		\coordinate (c3) at ($(c1)!.5!(c2)$);
		\node[above] at (c3 |- current bounding box.north) {\pgfplotslegendfromname{leg:pr_SDIRK3}};
	\end{tikzpicture}
	\caption{Convergence and order for third order DIRK schemes applied to the PR problem \cref{eq:pr_cos}.}
	\label{fig:pr_convergence_SDIRK3}
\end{figure}
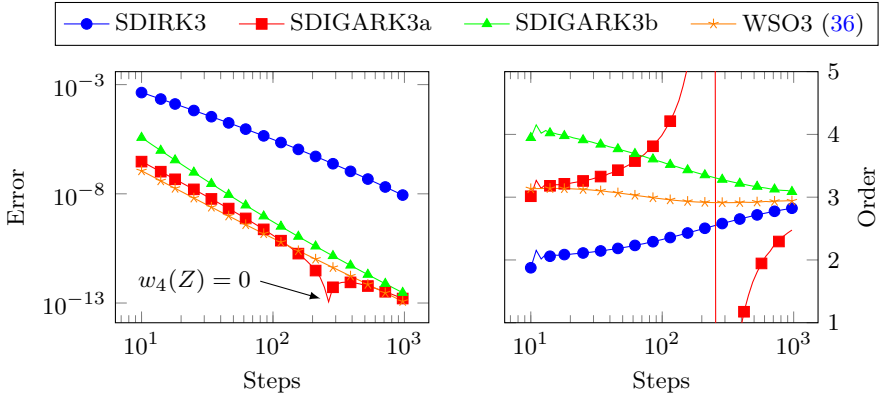

One way to avoid this behavior is by choosing coefficients such that the dominant error coefficient $W_4(z)$ is independent of $z$.  With an additional stage ($s\comp{2} = 5$), it is possible to enforce the additional constraint $w_{4,\ell} = 0$ for $\ell \geq 1$, and thus, $W_4(z) \equiv \frac{1}{24} - \frac{1}{6} \b\comp{2}* \c\comp{2}[\times 3]$.  Our updated method, SDIGARK3b, has a constant leading error term and the tableau
\begin{equation*}
	\begin{butchertableau}{cc|ccccc}
		\frac{\sqrt{3}+3}{6} & \frac{3-\sqrt{3}}{6} & \btentry{-3} & \btentry{-2} & \btentry{-1} & \btentry{0} & \btentry{1} \\ \hline
		\frac{\sqrt{3}+3}{6} & \btentry{0} & \frac{17 \sqrt{3}+29}{144} & \frac{-10 \sqrt{3}-17}{18} & \frac{73 \sqrt{3}+123}{72} & -\frac{11}{9}-\frac{5}{2 \sqrt{3}} & \frac{61 \sqrt{3}+109}{144} \\
		-\frac{1}{\sqrt{3}} & \frac{\sqrt{3}+3}{6} & \frac{-137 \sqrt{3}-243}{432} & \frac{79 \sqrt{3}+141}{54} & \frac{-187 \sqrt{3}-339}{72} & \frac{13}{3}+\frac{56}{9 \sqrt{3}} & \frac{-341 \sqrt{3}-507}{432} \\ \hline
		\frac{1}{2} & \frac{1}{2} & \frac{-5 (\sqrt{3}+2)}{72} & \frac{11 \sqrt{3}+23}{36} & \frac{-3 \sqrt{3}-7}{6} & \frac{13 \sqrt{3}+53}{36} & \frac{-7 (\sqrt{3}-2)}{72}
	\end{butchertableau}.
\end{equation*}
It maintains an order of at least three in \cref{fig:pr_convergence_SDIRK3}.  Also included in the figure is the convergence results for the WSO3 Runge--Kutta method
\begin{equation} \label{eq:DIRK3_WSO3}
	\begin{butchertableau}{c|cccc}
		\btentry{0.13756543551} & \btentry{0.13756543551} & \btentry{0} & \btentry{0} & \btentry{0} \\
		\btentry{0.80179011576} & \btentry{0.56695122794} & \btentry{0.23483888782} & \btentry{0} & \btentry{0} \\
		\btentry{2.33179673002} & \btentry{-1.08354072813} & \btentry{2.96618223864} & \btentry{0.44915521951} & \btentry{0} \\
		\btentry{0.59761291500} & \btentry{0.59761291500} & \btentry{-0.43420997584} & \btentry{-0.05305815322} & \btentry{0.88965521406} \\ \hline
		& \btentry{0.59761291500} & \btentry{-0.43420997584} & \btentry{-0.05305815322} & \btentry{0.88965521406}
	\end{butchertableau}.
\end{equation}
from \cite[page 458]{ketcheson2020dirk} which has order and weak stage order three.  While SDIGARK3a and SDIGARK3b have slightly larger errors than \cref{eq:DIRK3_WSO3}, they solve half as many linear systems and enjoy equal $\A\comp{1,1}_{i,i}$.

%% file: advection.tex
\section{Space-Time Convergence on a Hyperbolic PDE}
\label{sec:Order_Reduction:advection}

For a second numerical experiment, we will solve the following PDE used in \cite{sanz1986convergence}:
\begin{equation} \label{eq:advection_PDE}
	\begin{alignedat}{3}
		\pdv{u}{t} &= -\pdv{u}{x} + \frac{t - x}{(1 + t)^2}, \qquad & x, t &\in [0, 1], \\
		u(t, 0) &= \frac{1}{1 + t}, \qquad & t &\in [0, 1], \\
		u(0, x) &= 1 + x, \qquad & x &\in [0, 1].
	\end{alignedat}
\end{equation}
It possesses the simple solution $u(t,x) = (1+x)/(1+t)$.  We discretize in space with a first order, upwind finite difference scheme on the uniform grid $x_i = i h$, where $i = 0, \dots, \nvar$ and $h = \frac{1}{d}$.  This discretization is exact because the true solution is linear in space.  Note that $h$ is used as both the spatial grid size and the timestep in \cref{eq:gark_simplified}.  The semidiscretized form of \cref{eq:advection_PDE} is
\begin{equation} \label{eq:advection_ODE}
	y' = \begin{bmatrix}
		-\frac{1}{h} \\
		\frac{1}{h} & -\frac{1}{h} \\
		& \ddots & \ddots \\
		& & \frac{1}{h} & -\frac{1}{h}
	\end{bmatrix}
	y + \begin{bmatrix}
		\frac{t - x_1}{(1 + t)^2} + \frac{1}{h} \frac{1}{1 + t} \\
		\frac{t - x_2}{(1 + t)^2} \\
		\vdots \\
		\frac{t - x_{\nvar}}{(1 + t)^2}
	\end{bmatrix} \in \R{\nvar},
\end{equation}
and is of the form \cref{eq:linear_inhomogeneous_ode}.  We will examine the convergence as space and time are simultaneously refined.  We report the global errors computed in the $\ell^{\infty}$ norm at the final timestep: $\norm{e_d}_{\infty}$.

For the time discretization, we use the classical fourth order Runge--Kutta method (RK4) \cite{kutta1901beitrag}
\begin{equation} \label{eq:rk4}
	\begin{butchertableau}{c|cccc}
		\btentry{0} & \btentry{0} & \btentry{0} & \btentry{0} & \btentry{0} \\
		\frac{1}{2} & \frac{1}{2} & \btentry{0} & \btentry{0} & \btentry{0} \\
		\frac{1}{2} & \btentry{0} & \frac{1}{2} & \btentry{0} & \btentry{0} \\
		\btentry{1} & \btentry{0} & \btentry{0} & \btentry{1} & \btentry{0} \\ \hline
		& \frac{1}{6} & \frac{1}{3} & \frac{1}{3} & \frac{1}{6}
	\end{butchertableau},
\end{equation}
which by \cref{eq:lte_def} has the local truncation error
\begin{align*}
	\lte_n &= \frac{Z^3}{96} h^2 y''(t_n) + \frac{Z^3 - 2Z^2}{576} h^3 y^{(3)}(t_n) + \frac{Z^3-4 Z^2+8Z}{4608} h^4 y^{(4)}(t_n) \\
	& \quad + \frac{Z^3-6 Z^2+32 Z-16 \eye{\nvar}}{46080} h^5 y^{(5)}(t_n) + \order{Z^3 h^6}.
\end{align*}
If $Z = \order{h}$, we recover $\lte_n = \order{h^5}$ as expected.  For \cref{eq:advection_ODE}, however, $Z = \order{1}$ and the local truncation error is only $\order{h^2}$.  Starting with \cref{eq:rk4} as the base method, we can construct a GARK method \cref{eq:gark_simplified} that satisfies
\begin{equation*}
	w_{k,\ell} = 0, \text{ for } k = 0, \dots, 4 \text{ and } \ell = 0, \dots, 5.
\end{equation*}
to avoid order reduction.  With $s\comp{2} = 5$ and abscissae like that of a linear multistep method, we uniquely arrive at the following method which we will refer to as GARK4:
\begin{equation} \label{eq:rk4_gark}
	\begin{butchertableau}{cccc|ccccc}
		\btentry{0} & \frac{1}{2} & \frac{1}{2} & \btentry{1} & \btentry{-3} & \btentry{-2} & \btentry{-1} & \btentry{0} & \btentry{1} \\ \hline
		\btentry{0} & \btentry{0} & \btentry{0} & \btentry{0} & \btentry{0} & \btentry{0} & \btentry{0} & \btentry{0} & \btentry{0} \\
		\frac{1}{2} & \btentry{0} & \btentry{0} & \btentry{0} & \btentry{0} & \btentry{0} & \btentry{0} & \frac{1}{2} & \btentry{0} \\
		\btentry{0} & \frac{1}{2} & \btentry{0} & \btentry{0} & -\frac{1}{48} & \frac{1}{8} & -\frac{3}{8} & \frac{17}{24} & \frac{1}{16} \\
		\btentry{0} & \btentry{0} & \btentry{1} & \btentry{0} & -\frac{1}{16} & \frac{1}{3} & -\frac{5}{8} & \btentry{1} & \frac{17}{48} \\ \hline
		\frac{1}{6} & \frac{1}{3} & \frac{1}{3} & \frac{1}{6} & -\frac{5}{144} & \frac{13}{72} & -\frac{5}{12} & \frac{67}{72} & \frac{49}{144} \\
	\end{butchertableau}.
\end{equation}
GARK4 has the local truncation error
\begin{equation*}
	\lte_n = \frac{3 Z^3+17 Z^2+41 Z+12 \eye{\nvar}}{1440} h^5 y^{(5)}(t_n) + \order{Z^3 h^6},
\end{equation*}
and therefore, should not exhibit order reduction when applied to \cref{eq:advection_ODE}.  Indeed, this is verified in the convergence results presented in \cref{fig:rk4_convergence}.

\begin{figure}[ht!]
	\centering
	\begin{tikzpicture}
		\begin{groupplot}[group style={group size=2 by 1},xlabel={$\nvar = 1 / h$},width=0.48\textwidth]
			\nextgroupplot[xmode=log,ymode=log,ylabel={Absolute $\ell^{\infty}$ Error},legend to name=leg:rk4]
			\addplot table[x=Steps,y=RK Error,col sep=comma]{\datadir/advection.csv}; \addlegendentry{RK4}
			\addplot table[x=Steps,y=GARK Error,col sep=comma]{\datadir/advection.csv}; \addlegendentry{GARK4}
			\coordinate (c1) at (rel axis cs:0,1);
			\nextgroupplot[xmode=log,ylabel={Order},yticklabel pos=right]
			\addplot table[x=Steps,y=RK Order,col sep=comma]{\datadir/advection.csv};
			\addplot table[x=Steps,y=GARK Order,col sep=comma]{\datadir/advection.csv};
			\coordinate (c2) at (rel axis cs:1,1);
		\end{groupplot}
		\coordinate (c3) at ($(c1)!.5!(c2)$);
		\node[above] at (c3 |- current bounding box.north) {\pgfplotslegendfromname{leg:rk4}};
	\end{tikzpicture}
	\caption{Convergence and order for the methods \cref{eq:rk4,eq:rk4_gark} when applied to the advection problem \cref{eq:advection_ODE}.}
	\label{fig:rk4_convergence}
\end{figure}
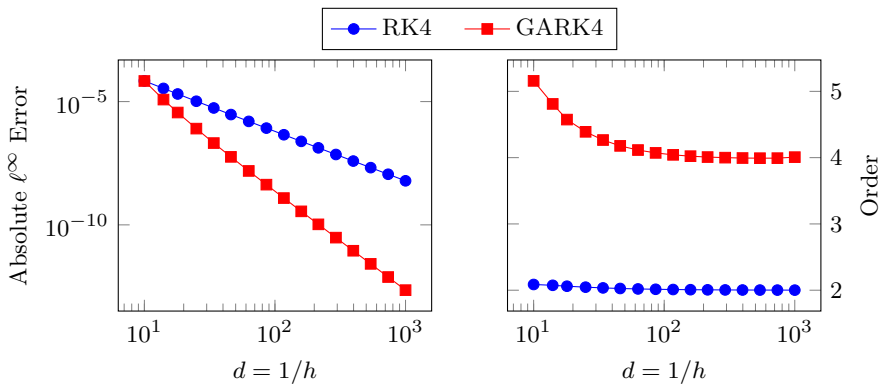

%% file: heat_eq.tex
\section{Time-Dependent Heat Equation Experiment}
\label{sec:Order_Reduction:heat_eq}

Our final experiment models the transient dynamics of heat in an aluminum heat sink via the PDE
\begin{subequations} \label{eq:heat_pde}
	\begin{alignat}{2}
		{\pdv{u}{t}}(t,\mathbf{x}) &= \frac{k}{c_p \rho} \laplacian{u}(t,\mathbf{x}), & \qquad \mathbf{x} &\in \Omega \subset \R{3}, \quad t \in [0, t_f], \\
		u(t, \mathbf{x}) &= T_\infty \left( 1 + 0.1 \sin\mleft(\frac{\pi t}{2 t_f}\mright) \right), & \qquad \mathbf{x} &\in \partial \Omega_{\text{bottom}} \\
		{\pdv{u}{\mathbf{n}}}(t,\mathbf{x}) &= \frac{h_c}{k} \left( u(t,\mathbf{x}) - T_\infty \right), & \qquad \mathbf{x} &\in \partial \Omega \setminus \partial \Omega_{\text{bottom}}, \\
		u(0, \mathbf{x}) &= T_\infty, & \qquad \mathbf{x} &\in \Omega.
	\end{alignat}
\end{subequations}
The domain $\Omega$ and snapshots of the solution are plotted in \cref{fig:heat_eq}.  The bottom face of the heat sink, $\Omega_{\text{bottom}}$, is in contact with a CPU and has a temperature specified by a time-dependent, Dirichlet boundary condition.  All other faces are in contact with the air and have convective, Robin boundary conditions.  Finally, the model's parameters are listed in \cref{tab:heat_eq_params}.
\begin{figure}[ht!]
	\centering
	\begin{subfigure}{0.49\textwidth}
		\includegraphics[width=\textwidth]{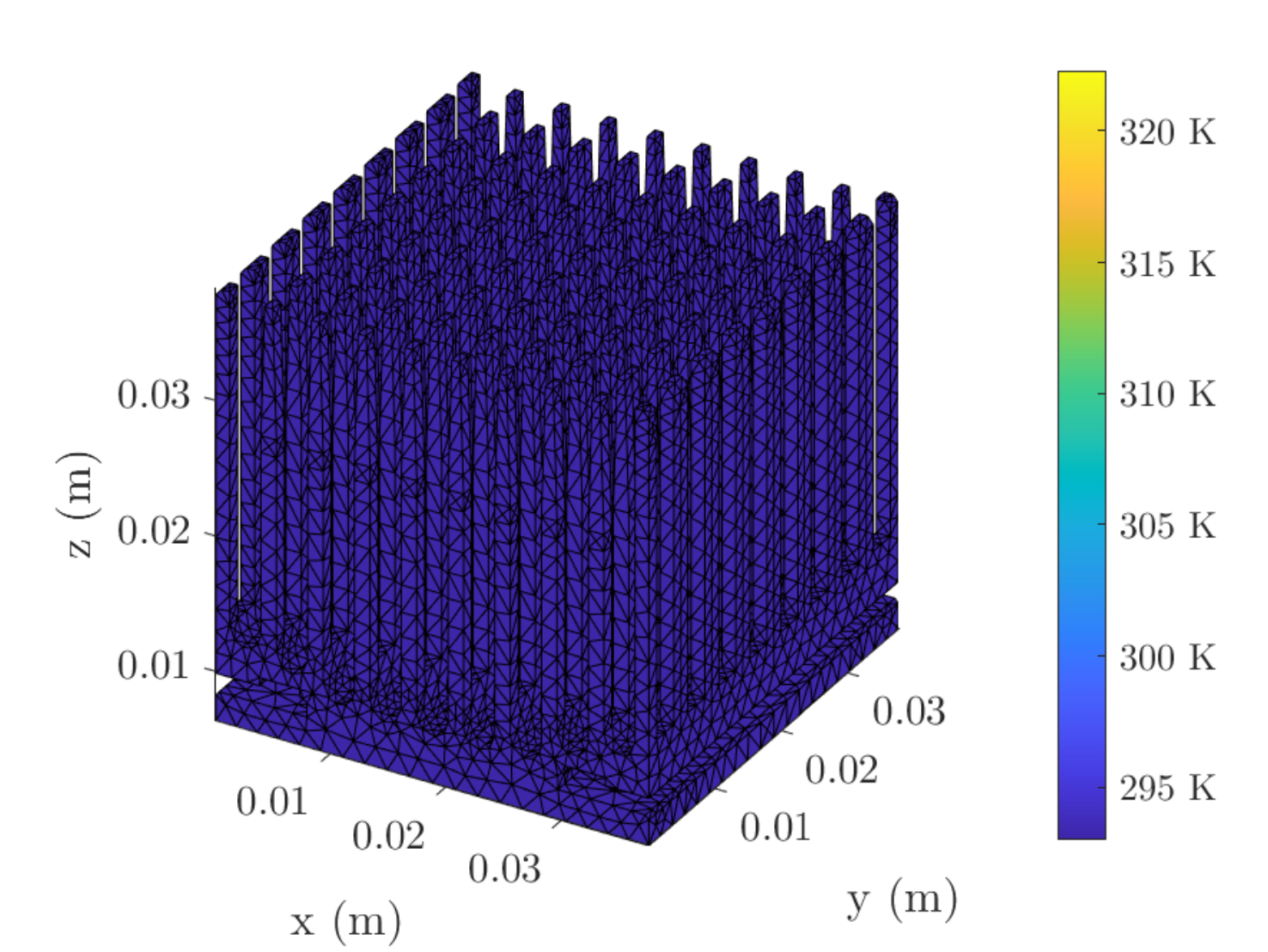}
		\caption{$t = 0$}
	\end{subfigure}
	\hfill
	\begin{subfigure}{0.49\textwidth}
		\includegraphics[width=\textwidth]{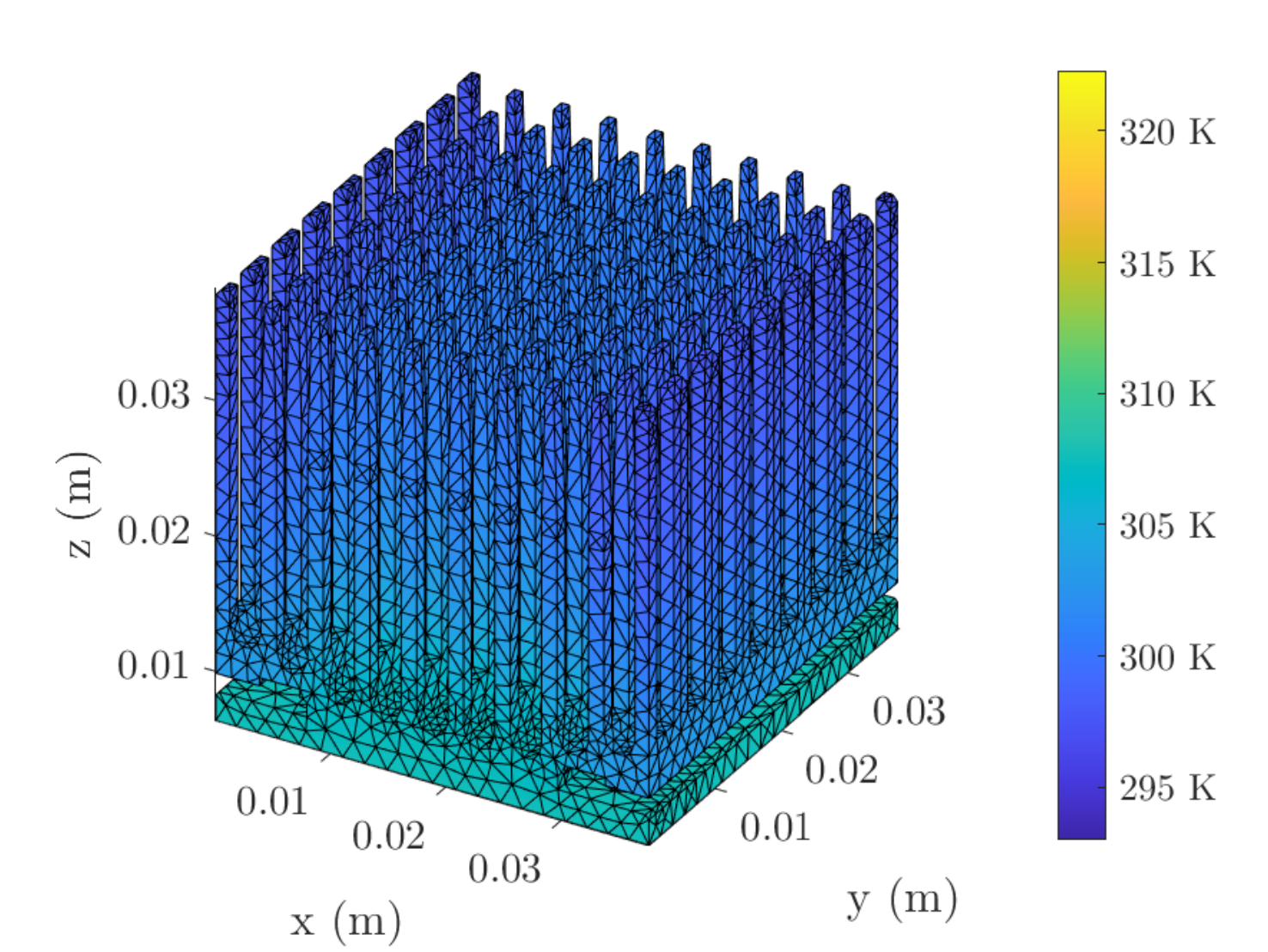}
		\caption{$t = 10$}
	\end{subfigure}
	\hfill
	\begin{subfigure}{0.49\textwidth}
		\includegraphics[width=\textwidth]{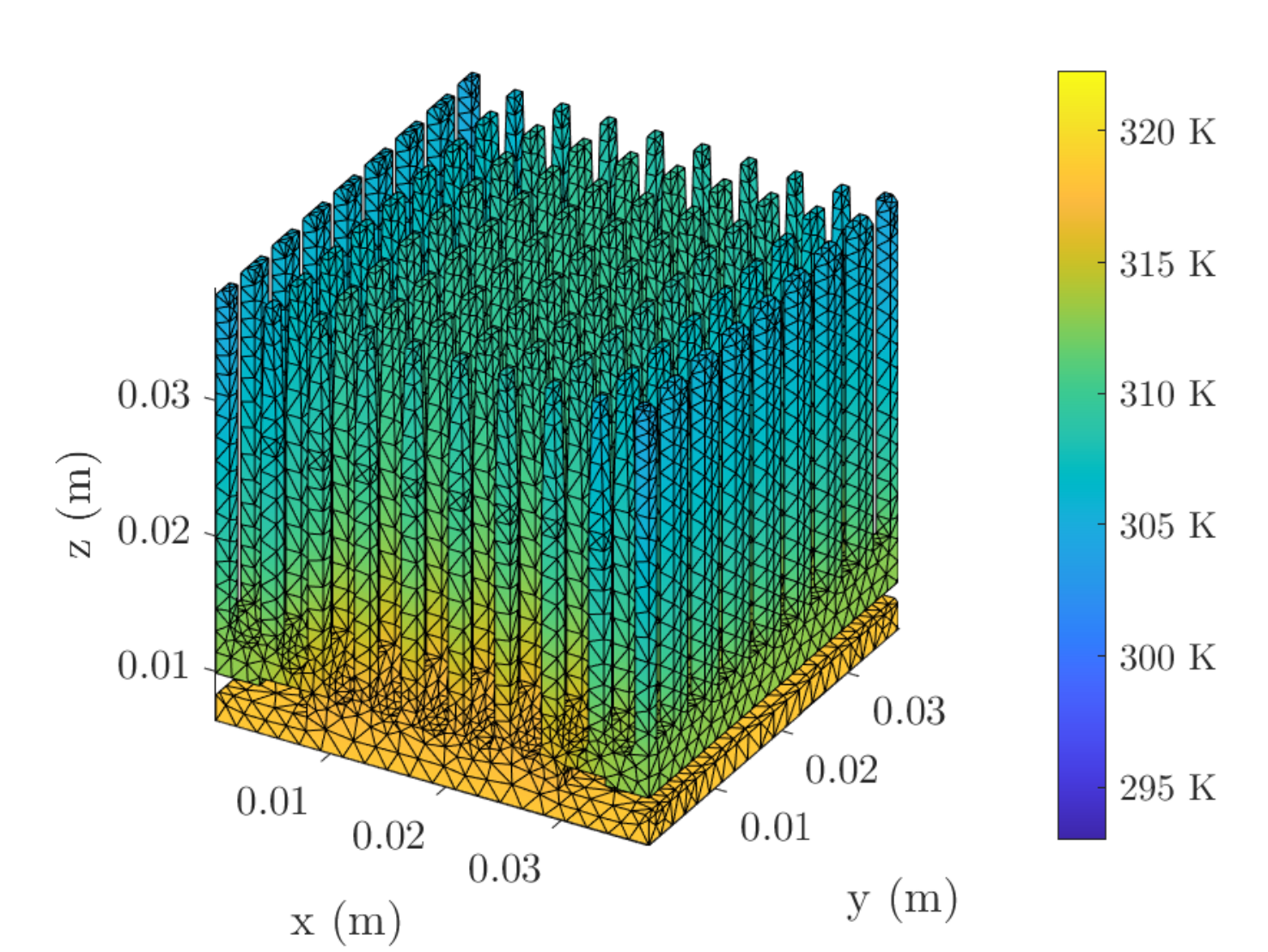}
		\caption{$t = 20$}
	\end{subfigure}
	\hfill
	\begin{subfigure}{0.49\textwidth}
		\includegraphics[width=\textwidth]{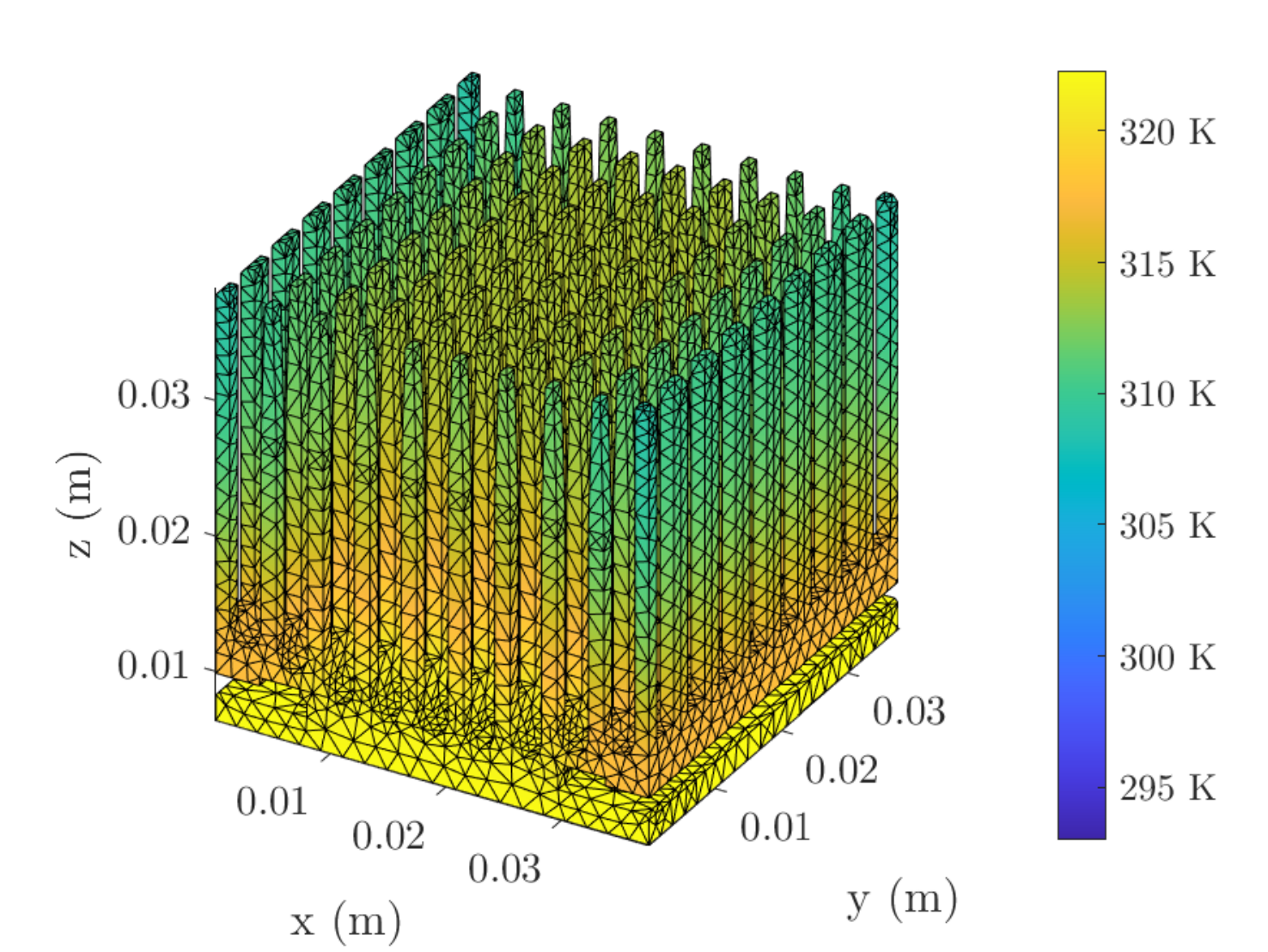}
		\caption{$t = 30$}
	\end{subfigure}
	\caption{Mesh and solution snapshots for the heat equation \cref{eq:heat_pde}.}
	\label{fig:heat_eq}
\end{figure}
\begin{table}[ht!]
	\centering
	\begin{tabular}{c|l|l}
		Variable & Description & Value \\ \hline
		$t_f$ & end time & \SI{30}{\second} \\
		$T_\infty$ & ambient air temperature & \SI{293}{\kelvin} \\
		$k$ & thermal conductivity & \SI{225.94}{\watt\per\meter\per\kelvin} \\
		$c_p$ & specific heat capacity & \SI{900}{\joule\per\kelvin\per\kilogram} \\
		$\rho$ & mass density & \SI{2698}{\kilogram\per\cubic\meter} \\
		$h_c$ & convective heat transfer coefficient & \SI{90}{\watt\per\square\meter\per\kelvin}
	\end{tabular}
	\caption{Parameters for heat equation \cref{eq:heat_pde}.}
	\label{tab:heat_eq_params}
\end{table}

Using MATLAB's PDE Toolbox, a second order, continuous finite element method is applied to the spatial dimensions of \cref{eq:heat_pde}.  The meshed heat sink contains 31139 elements and $\nvar = 65570$ degrees of freedom.  The resulting ODE is of the form \cref{eq:linear_inhomogeneous_ode} but with a mass matrix.

Fully implicit Runge--Kutta methods are some of the best-equipped to solve \cref{eq:linear_inhomogeneous_ode}, but even these are susceptible to order reduction.  For example, $s$-stage RadauIA methods have classical order $2s-1$ but stiff order $s-1$ for the PR problem \cite[Table 1]{prothero1974stability}.  Consider the third order RadauIA method
\begin{equation} \label{eq:radau}
	\begin{butchertableau}{c|cc}
		\btentry{0} & \frac{1}{4} & -\frac{1}{4} \\
		\frac{2}{3} & \frac{1}{4} & \frac{5}{12} \\ \hline
		& \frac{1}{4} & \frac{3}{4}
	\end{butchertableau},
	\qquad
	\lte_n = \frac{Z^2}{6} (Z^2-4 Z+6\eye{\nvar})^{-1} h^2 y''(t_n)
	+ \cdots.
\end{equation}
Following the same strategy used to derive SDIGARK3b in \cref{sec:Order_Reduction:pr:SDIRK_3}, we arrive at the following GARK extension to \cref{eq:radau}:
\begin{equation} \label{eq:radau_gark}
	\begin{butchertableau}{cc|ccccc}
		\btentry{0} & \frac{2}{3} & \btentry{-3} & \btentry{-2} & \btentry{-1} & \btentry{0} & \btentry{1} \\ \hline
		\frac{1}{4} & -\frac{1}{4} & -\frac{1}{81} & \frac{11}{162} & -\frac{17}{108} & \frac{53}{162} & -\frac{73}{324} \\
		\frac{1}{4} & \frac{5}{12} & -\frac{37}{972} & \frac{95}{486} & -\frac{137}{324} & \frac{389}{486} & \frac{32}{243} \\ \hline
		\frac{1}{4} & \frac{3}{4} & -\frac{11}{216} & \frac{7}{27} & -\frac{5}{9} & \frac{28}{27} & \frac{67}{216}
	\end{butchertableau},
	\qquad
	\lte_n = \frac{1}{72} h^4 y^{(4)}(t_n) + \cdots.
\end{equation}
The eigenvalues of $L$ for the discretized heat equation are all real and lie in the range $[-121314, -0.33]$.  For the range of $h$ used in our convergence experiments, the spectral radius of $Z$ can be as large as $40438$.  This is problematic for the local error of \cref{eq:radau}, and indeed, order reduction is exhibited in the convergence results of \cref{fig:heat_eq}.  With the GARK Radau IA scheme having a leading error term independent of $Z$, the global order of convergence is consistently three.

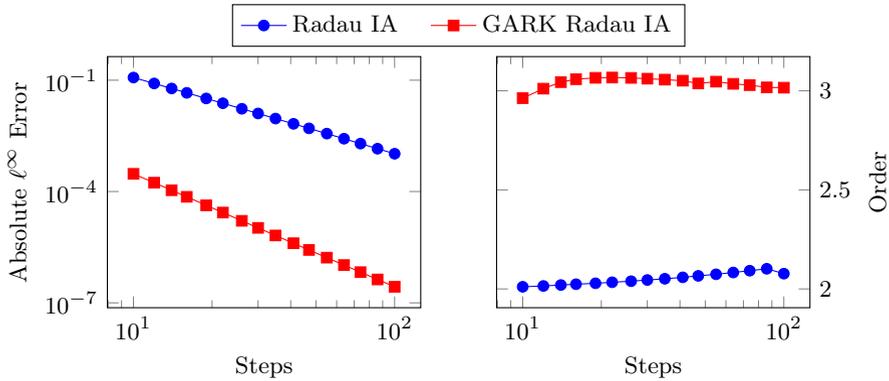
\begin{figure}[ht!]
	\centering
	\begin{tikzpicture}
		\begin{groupplot}[group style={group size=2 by 1},xlabel=Steps,width=0.48\textwidth]
			\nextgroupplot[xmode=log,ymode=log,ylabel={Absolute $\ell^{\infty}$ Error},legend to name=leg:radau]
			\addplot table[x=Steps,y=RadauIA Error,col sep=comma]{\datadir/heat.csv}; \addlegendentry{Radau IA}
			\addplot table[x=Steps,y=GARK RadauIAb Error,col sep=comma]{\datadir/heat.csv}; \addlegendentry{GARK Radau IA}
			\coordinate (c1) at (rel axis cs:0,1);
			\nextgroupplot[xmode=log,ylabel={Order},yticklabel pos=right]
			\addplot table[x=Steps,y=RadauIA Order,col sep=comma]{\datadir/heat.csv};
			\addplot table[x=Steps,y=GARK RadauIAb Order,col sep=comma]{\datadir/heat.csv};
			\coordinate (c2) at (rel axis cs:1,1);
		\end{groupplot}
		\coordinate (c3) at ($(c1)!.5!(c2)$);
		\node[above] at (c3 |- current bounding box.north) {\pgfplotslegendfromname{leg:radau}};
	\end{tikzpicture}
	\caption{Convergence and order for the methods \cref{eq:radau,eq:radau_gark} when applied to the heat equation \cref{eq:heat_pde}.}
	\label{fig:radau_convergence}
\end{figure} 

%% file: connections.tex
\section{Connections to Existing Analyses}
\label{sec:Order_Reduction:connections}

We now show how existing analyses of order reduction of Runge--Kutta methods can be viewed as special cases of the GARK analysis in this work.

If we set
\begin{equation} \label{eq:degenerate_gark}
\A\comp{1,1} = \A\comp{1,2} = A,
\quad
\b\comp{1} = \b\comp{2} = b,
\quad
\c\comp{1} = \c\comp{2} = c,
\end{equation}
the GARK method \cref{eq:gark_simplified} degenerates into the traditional Runge--Kutta method \cref{eq:rk}.  If we assume that this method has classical order $p$, the local error coefficients \cref{eq:lte_residuals} simplifiy to
\begin{equation} \label{eq:simplified_residual}
\begin{split}
W_0(z) &= 0 \\
W_k(z) &= 1 + \left( b^T  + z b^T  \left( \eye{s} - z A \right)^{-1} A \right) \left( z c^{k} - k c^{k-1} \right) \\
&= z b^T  \left( \eye{s} - z A \right)^{-1} \left( c^{k} - k A c^{k-1} \right), \qquad k \geq 1.
\end{split}
\end{equation}
With \cref{eq:degenerate_gark}, coefficients \eqref{eq:residual_coeffs} read
\begin{equation} \label{eq:simplified_residual_coeffs}
	w_{k,\ell} = \begin{cases}
		0, & k = 0,  \\
		1 - k b^T c^{k-1}, & k > 0, \ell = 0, \\
		b^T A ^{\ell-1} \left( c^{k}- k A  c^{k-1}\right), & k > 0, \ell \ge 1,
	\end{cases}
\end{equation}
and \eqref{eq:simplified_laurent_coeffs} becomes
\begin{equation} \label{eq:simplified_laurent_coeffs}
	x_{k,\ell} = \begin{cases}
		0, & k = 0, \ell \geq 0, \\
		0, & k \geq 0, \ell = -1, \\
		1 - b^T \boldsymbol{\Omega} c^{k}, & k > 0, \ell = 0, \\
		b^T \boldsymbol{\Omega}^{\ell+1} \left( c^{k} - k A c^{k-1}\right), & k > 0, \ell > 0,
	\end{cases}
\end{equation}
where we have used the fact that $b^T \boldsymbol{\Omega} A = b^T$.

The analysis of weak stage order conditions \cite[equation 3]{ketcheson2020dirk} defines the functions
\begin{equation*}
g^{(k)} = - \frac{1}{k} z b^T  \left( \eye{s} - z A \right)^{-1} \left( c^{k} - k A c^{k-1} \right), \qquad k \geq 1,
\end{equation*}
which match \cref{eq:simplified_residual} up to an inconsequential scaling.  Consider, for example, the Runge--Kutta method \cref{eq:DIRK3_WSO3} which has weak stage order three.  One can verify that $g^{(k)} \equiv W_k(z) \equiv 0$ for $k = 1, 2, 3$.  In fact, weak stage order $\widetilde{q}$ is equivalent to $g^{(k)} \equiv W_k(z) \equiv 0$ for $k = 1, \dots, \widetilde{q}$.

Ostermann and Roche use the functions
\begin{equation} \label{eq:Ostermann_W}
	W_k(z) = \frac{b^T (\eye{s} - z A)^{-1} (c^k - k A c^{k-1})}{1 - R(z)}, \qquad k \geq 1
\end{equation}
for the analysis of Runge--Kutta methods applied to linear, parabolic PDEs posed in Hilbert spaces \cite{ostermann1992runge}.  Again, order reduction can be mitigated by setting $W_k(z) \equiv 0$ for an appropriate set of $k$.  For small $z$, a series expansion of $W_k(z)$ is shown in \cite[page 406]{ostermann1992runge}. The requirement  $W_k(z) \equiv 0$ for $k = 1, \dots, p-1$ yields the order conditions
\begin{equation} \label{eq:classical_nonstiff}
	b^T A^{\ell} c^k - k b^T A^{\ell+1} c^{k-1} = 0, \qquad 0 \leq \ell \leq p - k - 1, \text{ and } 1 \leq k \leq p - 1.
\end{equation}
These correspond with our results in \cref{eq:simplified_residual_coeffs}.  The slightly different scaling of \cref{eq:Ostermann_W} compared to \cref{eq:simplified_residual} allows Ostermann and Roche to expand the global error in terms of $h^{k+1} W_k(Z) Z y^{(l)}(t)$ where $l = k, k+1$.  This is in contrast to the $h^k W_k(Z) y^{(k)}(t)$ we use.  Depending on the spectral properties of $Z$ and the choice of norm, $\nu$ can be rational in $h^{k+1} \norm{W_k(Z) Z y^{(l)}(t)} = \order{h^{\nu}}$.  With some care, the approach of Ostermann and Roche can be extended to GARK methods and can explain fractional orders of convergence.

The nonstiff order conditions \cref{eq:classical_nonstiff} also appear in the global error analysis of Runge--Kutta methods applied to the PR problem \cite[page 108]{rang2014analysis}.  Further, Rang expands the global error about $z = \infty$ to derive stiff order conditions \cite[equations 20 and 21]{rang2014analysis}.  These match our $x_{k,\ell}$ coefficients in \cref{eq:simplified_laurent_coeffs}.

Finally, we note there is a strong connection between our GARK-based approach and the technique of modifying boundary conditions within Runge--Kutta stages to eliminate order reduction \cite{abarbanel1996removal,pathria1997correct,alonso2002runge,alonso2004avoiding}.  Recall, for example, GARK4 from \cref{eq:rk4_gark}.  One can verify with a Taylor expansion that
\begin{equation*}
	\left( \A\comp{1,2} \otimes \eye{\nvar} \right) g \mleft( t_n + \c\comp{2} h \mright) = \begin{bmatrix}
		0 \\
		\frac{1}{2} g(t_n) \\
		\frac{1}{2} g(t_n) + \frac{h}{4} g'(t_n) \\
		g(t_n) + \frac{h}{2} g'(t_n) + \frac{h^2}{4} g''(t_n)
	\end{bmatrix} + \order{h^5}.
\end{equation*}
These are exactly the modified boundary conditions given in \cite[equations 2.10 to 2.12]{abarbanel1996removal} for RK4.  Conversely, it is possible to convert modified boundary conditions into the companion method of \cref{eq:gark_simplified} by ``undoing'' the Taylor series.  We can see both techniques control the local truncation error in a similar manner but differ in the treatment of $g$: linear combinations at different times versus linear combinations of derivatives.

%% file: conclusion.tex
\section{Conclusions}
\label{sec:Order_Reduction:conclusion}

Even on simple, linear ODEs, Runge--Kutta methods are susceptible of order reduction: the actual order of convergence may be lower than that predicted by the order condition theory. In the last several decades, studies into B-convergence and stiff order conditions have addressed issues of order reduction but in ways that are often expensive.  The solutions involve additional order conditions, which often require additional stages or more coupling among stages, which increases the computational costs.

This work develops an inexpensive approach to avoiding order reduction in linear problems that only changes the number of forcing evaluations.  The GARK framework has provided the necessary foundation to couple Runge--Kutta methods, possibly with differing numbers of stages, for the linear and forcing terms.  Our approach contains previous works on alleviating order reduction as special cases.  We have presented an error analysis that makes no assumptions on the dependence of $Z$ on $h$ and derived conditions to ensure convergence independent of the stiffness.  Finally, our numerical experiments have shown the computational effectiveness of the new schemes on problems like the scalar PR problem as well as more challenging PDEs.  There are several possible extensions to this work including nonlinear problems and application of implicit-explicit (IMEX) methods.

%% file: appendix.tex
\section{Proof of Classical Order Conditions}
\label{app:classical_order_conditions}

\begin{proof}[Proof of \cref{tmh:classical_order_conditions}]
	In order to use N-tree order condition theory \cite{araujo1997symplectic,sandu2015generalized}, we switch to an autonomous form of \cref{eq:ode_split}:
	\begin{equation} \label{eq:ode_autonomous}
		\widetilde{y}' = \underbrace{\begin{bmatrix}
			L & 0 \\ 0 & 0
			\end{bmatrix} \widetilde{y}}_{\widetilde{f}\comp{1}(\widetilde{y})} + \underbrace{\begin{bmatrix}
			g(t) \\ 0
			\end{bmatrix}}_{\widetilde{f}\comp{2}(\widetilde{y})} + \underbrace{\begin{bmatrix}
			0 \\ 1
			\end{bmatrix}}_{\widetilde{f}\comp{3}(\widetilde{y})},
		\qquad
		\widetilde{y} = \begin{bmatrix}
			y \\ t
		\end{bmatrix} \in \R{\nvar + 1}.
	\end{equation}
	Note that the original method \cref{eq:gark_simplified} is equivalent to the GARK scheme
	\begin{equation*}
		\begin{butchertableau}{c|c|c}
			\A\comp{1,1} & \A\comp{1,2} & \c\comp{1} \\ \hline
			\A\comp{2,1} & \A\comp{2,2} & \c\comp{2} \\ \hline
			\b\comp{1}* & \b\comp{2}* & 1 \\ \hline
			\b\comp{1}* & \b\comp{2}* & 1
		\end{butchertableau}
	\end{equation*}
	applied to \cref{eq:ode_autonomous}.  This three-partitioned method is order $p$ if and only if
	\begin{equation*}
		\sum_{\substack{\tree \in T_3\\\rho(\tree) \leq p}} \left( \Phi(\tree) - \frac{1}{\gamma(\tree)} \right) F(\tree)(\widetilde{y}) = 0,
	\end{equation*}
	where $T_3$ is the set of three-trees, and $\rho$, $\Phi$, $\gamma$ are the order, elementary weight, and density of a tree, respectively.  Tree vertices for partitions one, two, and three are represented by $\btree{[1]}$, $\btree{[2]}$, and $\btree{[3]}$, respectively.  The elementary differentials for \cref{eq:ode_autonomous} simplify to
	\begin{align*}
		F\mleft( \tree \mright)(\widetilde{y}) &= \begin{cases}
			\widetilde{y}, & \text{if } \tree = \emptyset, \\
			\begin{bmatrix}
				0 \\ 1
			\end{bmatrix}, & \text{if } \tree = \btree{[3]}, \\
			\begin{bmatrix}
				L & 0 \\ 0 & 0
			\end{bmatrix} F(\utree)(\widetilde{y}), & \text{if } \tree = \btree{[1[$\utree$]]}, \text{ where } \utree \in T_3, \\
			\begin{bmatrix}
				g^{(m)}(t) \\ 0
			\end{bmatrix}, & \text{if } \tree = \overbrace{\btree{[2[3][.][3]]}}^{m}, \text{ where } m \geq 0, \\
			\zero{\nvar+1}, & \text{otherwise}.
		\end{cases}
	\end{align*}
	For the elementary differentials that do not vanish, we split their corresponding trees into four sets.  The first are trees where all vertices are $\btree{[1]}$ and singly-branched.  By considering their GARK elementary weights and densities, we recover the order conditions \cref{eq:nonstiff_order_conditions:linear}.  Similarly, bushy trees with $\btree{[2]}$ as the root and $\btree{[3]}$ for the leaves yield \cref{eq:nonstiff_order_conditions:bushy}.  Next, trees of the form
	\begin{equation*}
		\btree{[1[:[1[2]]]]}
		\quad \text{and} \quad
		\btree{[1[:[1[2[3][.][3]]]]]}
	\end{equation*}
	lead to \cref{eq:nonstiff_order_conditions:palm}.  Finally, the tree $\btree{[3]}$ corresponds to a trivially true order condition because $\b\comp{3} = \begin{bmatrix} 1 \end{bmatrix}$.
\end{proof}